\documentclass[12pt]{amsart}

\usepackage{enumitem}
\usepackage{psfrag}
\usepackage{graphicx}
\usepackage{pinlabel}
\usepackage{graphicx}
\usepackage{amsmath}
\usepackage{amssymb}
\usepackage{amscd}
\usepackage{autobreak}
\usepackage{picinpar}
\usepackage{times}
\usepackage{pb-diagram}
\usepackage{graphicx}
\usepackage{wrapfig}
\usepackage{xspace}
\usepackage{hyperref}
\usepackage{url}
\usepackage{caption}
\usepackage{subcaption}
\usepackage{fancyhdr}
\usepackage{overpic}
\usepackage{color}
\usepackage[square,comma,sort&compress,numbers]{natbib}
\usepackage{latexsym}
\usepackage{mathrsfs}
\usepackage{amsfonts}
\usepackage{hyperref}
\usepackage{amsthm}
\usepackage{appendix}
\usepackage{pdfsync}
\usepackage{newtxmath}

\theoremstyle{definition}
\newtheorem{theorem}{Theorem}[section]
\newtheorem{lemma}[theorem]{Lemma}
\newtheorem{proposition}[theorem]{Proposition}
\newtheorem{corollary}[theorem]{Corollary}
\newtheorem{definition}[theorem]{Definition}
\newtheorem{conjecture}[theorem]{Conjecture}
\newtheorem{remark}[theorem]{Remark}
\newtheorem*{theorem*}{Theorem}

\DeclareMathOperator{\adj}{adj}
\DeclareMathOperator{\diag}{diag}

\graphicspath{{figures/}}   %% 图片路径. 本文的图片都放在这个文件夹里了.

\def\qed{\hfill{Q.E.D.}\smallskip}

\begin{document}

\title{\bf Thurston's sphere packing on 3-dimensional manifolds, II}
\author{Xiaokai He, Xu Xu, Chao Zheng}

\date{\today}

\address{School of Mathematics and Statistics, Hunan First Normal University, Changsha 410205, P.R. China} 
\email{hexiaokai77@163.com}

\address{School of Mathematics and Statistics, Wuhan University, Wuhan, 430072, P.R. China} \email{xuxu2@whu.edu.cn}

\address{International Center for Mathematical Research (BICMR), Beijing (Peking) University, Beijing, 100871, P.R. China}
\email{czheng@whu.edu.cn}



\begin{abstract}
In this paper, we study the rigidity of Thurston's hyperbolic sphere packings on 3-dimensional manifolds. 
We prove that Thurston's hyperbolic sphere packing is locally determined by combinatorial scalar curvature. 
We further prove the infinitesimal rigidity that Thurston's hyperbolic sphere packing can not be deformed while keeping the combinatorial Ricci curvature fixed. The main tools include a characterization of the admissible space of Thurston's hyperbolic sphere packings, 
a variational formula for the dihedral angles of a tetrahedron and an identity for the variations of dihedral angles in a tetrahedron generated by Thurston's hyperbolic sphere packings.
\end{abstract}

\maketitle

\section{Introduction}

This work continues the research initiated by He and Xu in \cite{He-Xu}, where the authors investigated the rigidity of Thurston’s Euclidean sphere packings on 3-dimensional manifolds. In the present paper, we focus on Thurston’s hyperbolic sphere packings on 3-manifolds. These packings constitute three-dimensional analogues of Thurston’s hyperbolic circle packings on surfaces, which have been extensively studied in \cite{CL,H2,St,T1,Z} and other related works. Meanwhile, they serve as generalizations of tangential sphere packings on closed 3-manifolds, a topic explored in \cite{CR,G thesis,G1,G2,Glickenstein JDG,Thomas} and prior literature.

Suppose $M$ is a 3-dimensional connected closed manifold with a triangulation $\mathcal{T}=\{V,E,F,T\}$,
where $V, E, F, T$ represent the sets of vertices, edges, faces and tetrahedra, respectively. 
We use $(M,\mathcal{T})$ to denote a triangulated 3-manifold, and
use $i$, $\{ij\}$, $\{ijk\}$, $\{ijkh\}$ to denote a vertex, edge, face and tetrahedron in $\mathcal{T}$ respectively.
We sometimes use $\sigma$ to denote a tetrahedron in $\mathcal{T}$.

\begin{definition}\label{Def: SP metric}
Suppose $(M,\mathcal{T})$ is a triangulated connected closed 3-dimensional manifold.
Let $\Phi: E\rightarrow [0,\frac{\pi}{2}]$ be a weight defined on the edges of $(M,\mathcal{T})$.
Thurston's hyperbolic sphere packing metric on $(M,\mathcal{T},\Phi)$ is defined as a map $r: V\rightarrow (0,+\infty)$ such that
\begin{description}
\item[(a)] 
The edge length $l_{ij}$ of the edge $\{ij\}\in E$ between two vertices $i$ and $j$ is given by
\begin{equation}\label{Eq: Thurston hyperbolic length}
\cosh l_{ij}=\cosh r_{i}\cosh r_j+\sinh r_i\sinh r_j\cos \Phi_{ij};
\end{equation}
\item[(b)] 
The edge lengths $l_{ij}, l_{ik}, l_{ih}, l_{jk}, l_{jh}, l_{kh}$ determine a non-degenerate hyperbolic tetrahedron for each tetrahedron $\{ijkh\}\in T$.
\end{description}
\end{definition}

The function $l: E\rightarrow (0, +\infty)$ determined by (\ref{Eq: Thurston hyperbolic length}) is a discrete hyperbolic metric on $(M,\mathcal{T})$.
The condition (b) in Definition \ref{Def: SP metric} is called the non-degenerate condition.
In the following, we sometimes refer to a Thurston's hyperbolic sphere packing metric in Definition \ref{Def: SP metric} as a non-degenerate Thurston's hyperbolic sphere packing metric.
If $\Phi\equiv0$, then Thurston's hyperbolic sphere packing metric is exactly Cooper-Rivin's tangential sphere packing metric in \cite{CR}.
Given a Thurston's hyperbolic sphere packing metric $r$ on $(M,\mathcal{T},\Phi)$,
we attach to each vertex \(i\in V\) a sphere \(B_i\) of radius \(r_i\). 
For vertices \(i,j\in V\), the intersection angle of \(B_i\) and \(B_j\) is \(\Phi_{ij}\). 
Then the length \(l_{ij}\) of the edge \(\{ij\}\) is a function of \(r_i\), \(r_j\) and the fixed angle \(\Phi_{ij}\), given by the hyperbolic cosine law as in \eqref{Eq: Thurston hyperbolic length}. 
See Figure \ref{triangle}.
%a sphere $B_i$ of radius $r_i$ is attached to each vertex $i\in V$, and the intersection angle of $B_i$ and $B_j$ attached to the vertices $i,j\in V$ is $\Phi_{ij}$.
%Then the length $l_{ij}$ of the edge $\{ij\}$ is a function of $r_i,\ r_j$ and the fixed angle $\Phi_{ij}$ defined by (\ref{Eq: Thurston hyperbolic length}) by the cosine law for hyperbolic triangles. 
\begin{figure}[!ht]
\centering
\begin{overpic}[scale=1.5]{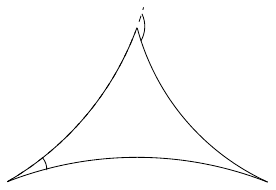}
\put (-0.5,-0.5) {$i$}
\put(97,-0.5) {$j$}
\put (64,34) {$r_j$} 
\put (30,34) {$r_i$} 
\put (18,10) {$\theta_{ij}$}
\put (54,56) {$\Phi_{ij}$}
\end{overpic}
\caption{A hyperbolic triangle.}
  \label{triangle}
\end{figure}

For each tetrahedron $\sigma\in T$, Thurston's hyperbolic sphere packing metric determines a non-degenerate hyperbolic tetrahedron by Definition \ref{Def: SP metric}.
Gluing these hyperbolic tetrahedra along the faces by isometries gives rise to a piecewise hyperbolic 3-dimensional manifold,
which has singularities along the edges and vertices.
We use the following combinatorial Ricci curvature to describe the singularities along the edges.

\begin{definition}\label{Def: CRC}
Suppose $(M,\mathcal{T})$ is a triangulated connected closed 3-dimensional manifold.
Let $l: E\rightarrow(0,+\infty)$ be a discrete hyperbolic metric on $(M,\mathcal{T})$.
The combinatorial Ricci curvature of $l$ along the edge $\{ij\}\in E$ is defined as
\begin{equation}\label{Eq: CRC}
K_{ij}=2\pi-\sum_{\{ijkh\}\in T}\beta_{ij,kh},
\end{equation}
where $\beta_{ij,kh}$ is the dihedral angle along the edge $\{ij\}$ in a tetrahedron $\{ijkh\}$, and the summation is taken over the tetrahedra with $\{ij\}$ as a common edge.
\end{definition}

Similar to the combinatorial scalar curvature for Thurston's Euclidean sphere packings on 3-dimensional closed manifolds \cite{Glickenstein JDG, He-Xu}, 
we use the following combinatorial scalar curvature for Thurston's hyperbolic sphere packings to describe the singularities at the vertices.

\begin{definition}\label{Def: CSC}
Suppose $(M,\mathcal{T})$ is a triangulated connected closed 3-dimensional manifold.
Let $\Phi: E\rightarrow [0,\frac{\pi}{2}]$ be a weight defined on the edges of $(M,\mathcal{T})$.
The combinatorial scalar curvature at a vertex $i\in V$ for Thurston's hyperbolic sphere packing metric $r$ is defined as
\begin{equation*}%\label{Eq: CSC}
K_i=\sum_{j;j\sim i}K_{ij}\cos \theta_{ij},
\end{equation*}
where $\theta_{ij}$ is the inner angle opposite to the edge with length $r_j$ in the hyperbolic triangle formed by three edges with lengths $r_i,\ r_j,\ l_{ij}$, as shown in Figure \ref{triangle}.
\end{definition}

If $\Phi_{ij}=0$ for any edge $\{ij\}\in E$, then $\cos \theta_{ij}=1$ and hence
\begin{equation*}
K_i=\sum_{j;j\sim i}\bigg(2\pi-\sum_{\{ijkh\}\in T}\beta_{ij,kh}\bigg)
=4\pi-\sum_{\{ijkh\}\in T} \alpha_{i,jkl},
\end{equation*}
where $\alpha_{i,jkl}$ is the solid angle at the vertex $i$ of the tetrahedron $\{ijkh\}\in T$.
This is the combinatorial scalar curvature introduced by Cooper-Rivin \cite{CR}.

We have the following local rigidity for Thurston's hyperbolic sphere packing metrics with respect to the combinatorial scalar curvature in Definition \ref{Def: CSC}.

\begin{theorem}\label{Thm: local rigidity}
Suppose $(M,\mathcal{T},\Phi)$ is a triangulated connected closed 3-dimensional manifold with a weight $\Phi: E\rightarrow [0,\frac{\pi}{2}]$ satisfying
\begin{equation}\label{Eq: SC1}
\Phi_{ij}+\Phi_{jk}+\Phi_{ik}>\pi,\quad \forall\{ijk\}\in F,
\end{equation}
or
\begin{equation}\label{Eq: SC2}
\Phi\equiv C\in [0,\frac{\pi}{2}].
\end{equation}
If $\overline{r}$ is a non-degenerate Thurston's hyperbolic sphere packing metric on $(M,\mathcal{T},\Phi)$ satisfying $K_{ij}(\overline{r})\sin^2\Phi_{ij}\geq0$ for any edge $\{ij\}\in E$,
then there exists a neighborhood $U$ of $\overline{r}$ such that Thurston's hyperbolic sphere packing metric in $U$ is determined by its combinatorial scalar curvature.
\end{theorem}

As a direct consequence of Theorem \ref{Thm: local rigidity}, we have the following corollary.
\begin{corollary}
Suppose $(M,\mathcal{T},\Phi)$ is a triangulated connected closed 3-dimensional manifold with a weight $\Phi: E\rightarrow [0,\frac{\pi}{2}]$ satisfying (\ref{Eq: SC1}) or (\ref{Eq: SC2}).
If $\overline{r}$ is a non-degenerate Thurston's hyperbolic sphere packing metric on $(M,\mathcal{T},\Phi)$ with $K_{ij}(\overline{r})\geq0$ for any edge $\{ij\}\in E$,
then there exists a neighborhood $U$ of $\overline{r}$ such that Thurston's hyperbolic sphere packing metric in $U$ is determined by its combinatorial scalar curvature.
\end{corollary}

\begin{remark}
In the special case $\Phi\equiv 0$, the condition $K_{ij}(\overline{r})\sin^2\Phi_{ij}\geq0$ has no restrictions on the combinatorial Ricci curvature $K_{ij}$. In this case, Theorem \ref{Thm: local rigidity} reduces to the local rigidity of tangential hyperbolic sphere packings on 3-dimensional manifolds, a result claimed by Cooper-Rivin in \cite{CR}. In this paper, we give a rigorous proof of this fact as a by product of Theorem \ref{Thm: local rigidity}. The global rigidity of the tangential hyperbolic sphere packings has been proved in \cite{X1}.
\end{remark}

We further have the following infinitesimal rigidity for Thurston's hyperbolic sphere packings with respect to the combinatorial Ricci curvature in Definition \ref{Def: CRC}.

\begin{theorem}\label{Thm: infinitesimal rigidity}
Suppose $(M,\mathcal{T},\Phi)$ is a triangulated connected closed 3-dimensional manifold with a weight $\Phi: E\rightarrow [0,\frac{\pi}{2}]$ satisfying (\ref{Eq: SC1}) or (\ref{Eq: SC2}).
Thurston's hyperbolic sphere packing metric on $(M,\mathcal{T},\Phi)$ can not be deformed while keeping the combinatorial Ricci curvatures along the edges fixed.
\end{theorem}

\begin{remark}
As mentioned in \cite{He-Xu}, there is a subtle difference between the notion of the local rigidity in Theorem \ref{Thm: local rigidity} and the infinitesimal rigidity in Theorem \ref{Thm: infinitesimal rigidity}.
In fact, the local rigidity implies the infinitesimal rigidity, but the reverse is not true due to dimensional differences.
\end{remark}

The basic idea underlying the proofs of Theorem \ref{Thm: local rigidity} and Theorem \ref{Thm: infinitesimal rigidity} is to use a variational principle with the discrete Hilbert-Einstein functional as the action functional. 
There are two main difficulties in the proof. 
The first difficulty arises from the connectedness of the admissible space of Thurston's hyperbolic sphere packing metrics for a tetrahedron. We resolve this difficulty in Theorem \ref{Thm: simply connect}. 
The main idea of the proof comes from a new proof \cite{X3} of Bowers-Stephenson's conjecture on the rigidity of inversive distance circle packings on surfaces \cite{X2}, 
which simplifies the proofs of Bowers-Stephenson's conjecture in \cite{Guo, L2, X2}. 
The second difficulty concerns the negative definiteness of the matrix $\Lambda_\sigma$ in (\ref{eq:matrices}).
To overcome it, we first derive an explicit variational formula for the dihedral angles of a non-degenerate tetrahedron arising from Thurston's hyperbolic sphere packings, and then establish a useful identity for the variations of these dihedral angles.

Note that the rigidity of Thurston's sphere packings on closed 3-manifolds are not fully understood in Theorem \ref{Thm: local rigidity} and Theorem \ref{Thm: infinitesimal rigidity}. 
For example,  the tangential sphere packings on closed 3-manifolds have been proved to be globally rigid \cite{X1}.
It is natural to ask whether the global rigidity holds for all Thurston's hyperbolic sphere packing metrics on closed 3-manifolds. 
Motivated by Theorem \ref{Thm: local rigidity}, Theorem \ref{Thm: infinitesimal rigidity} and the global rigidity of the tangential sphere packings in \cite{X1}, 
we have the following conjecture on the rigidity of Thurston's hyperbolic sphere packings on closed 3-manifolds.

\begin{conjecture}
Suppose $(M,\mathcal{T},\Phi)$ is a triangulated connected closed 3-dimensional manifold with a weight $\Phi: E\rightarrow [0,\frac{\pi}{2}]$.
In the space of Thurston's hyperbolic sphere packing metrics on $(M,\mathcal{T},\Phi)$ with $K_{ij}(\overline{r})\sin^2\Phi_{ij}\geq0$ for any edge $\{ij\}\in E$, Thurston's hyperbolic sphere packing metric is determined by the combinatorial scalar curvature.
\end{conjecture}

The paper is organized as follows.
In Section \ref{section 2}, we study the admissible space of Thurston's hyperbolic sphere packing metrics for a tetrahedron. As a result, we prove Theorem \ref{Thm: simply connect}.
In Section \ref{section 3}, we derive the variational formula of the dihedral angles and prove a useful identity of the variational formulas. As a result, we prove the key negativeness of the matrix $\Lambda_\sigma$ appearing in the variation of the discrete Hilbert-Einstein functional, under a generalization of the conditions (\ref{Eq: SC1}) and (\ref{Eq: SC2}).
In Section \ref{section 4}, we prove generalizations of Theorem \ref{Thm: local rigidity} and Theorem \ref{Thm: infinitesimal rigidity}.
~\\
\\
\textbf{Acknowledgment} 
The work was started in 2018 when the second author Xu Xu was visiting Rutgers University. Xu Xu thanks Professor Feng Luo for his invitation to Rutgers University and communications and interesting on this work, and thanks Professor Tian Yang for communications during the visiting. 
The research of Xiaokai He is partially supported by the National Natural Science Foundation of China under Grant No. 12475049.
The research of Xu Xu is supported by the National Natural Science Foundation of China under Grant No. 12471057 and Hubei Provincial Natural Science Fund for Outstanding Research
Groups under Grant No. 2026AFA008. The research of Chao Zheng is supported by the China National Postdoctoral Program for Innovative Talents under grant no. BX20250068.
\\
\\
\noindent\textbf{AI Statement} The proof of Lemma \ref{Lem: 2} in this paper was first accomplished with the help of AI. We then check it with Wolfram Mathematica 12. For readers' convenience, we supply all the Mathematica codes. 
%The remaining text was polished with the help of AI.

\section{Admissible space of Thurston's hyperbolic sphere packing metrics}\label{section 2}

Suppose $\sigma=\{1234\}$ is a hyperbolic tetrahedron with edge lengths $l_{12}, l_{13}, l_{14}, l_{23}, l_{24}, l_{34}$ defined by (\ref{Eq: Thurston hyperbolic length}).
The admissible space $\Omega_{1234}(\Phi)$ of Thurston's hyperbolic sphere packing metrics $r$ on $\sigma=\{1234\}$ is the set of $(r_1, r_2, r_3, r_4)\in \mathbb{R}^4_{>0}$ such that the hyperbolic tetrahedron $\sigma=\{1234\}$ is non-degenerate.
Set
\begin{equation}\label{Eq: G}
\begin{aligned}
G(l)=(\cosh l_{ij})_{4\times 4}=\left(
\begin{array}{ccccc}
1 & \cosh l_{12} & \cosh l_{13} & \cosh l_{14}\\
\cosh l_{12} & 1 & \cosh l_{23} & \cosh l_{24}\\
\cosh l_{13} & \cosh l_{23} & 1 & \cosh l_{34}\\
\cosh l_{14} & \cosh l_{24} & \cosh l_{34} & 1 \\
\end{array}
\right).
\end{aligned}
\end{equation}
Here we set $l_{ii}=0$.

\begin{proposition}[\cite{TW}, Proposition 2.4.1]\label{Prop: non-degenerate condition}
A non-degenerate hyperbolic tetrahedron $\sigma=\{1234\}$ with edge lengths $l_{12}, l_{13}, l_{14}, l_{23}, l_{24}, l_{34}$ exists
if and only if all principal minors of $-G$ are negative.
\end{proposition}

Note that all principal minors of $-G$ being negative, together with $\det G<0$, is equivalent to the satisfaction of the triangle inequalities for all faces of the hyperbolic tetrahedron.
For a hyperbolic tetrahedron $\sigma=\{1234\}$ whose edge lengths are given by (\ref{Eq: Thurston hyperbolic length}), the triangle inequalities on its faces hold, as established in \cite{T1,X2,X3,Z}.
By Proposition \ref{Prop: non-degenerate condition}, to verify that $l_{12}, l_{13}, l_{14}, l_{23}, l_{24}, l_{34}$ define a non-degenerate hyperbolic tetrahedron, it suffices to consider the condition $\det G<0$.

In this section, to simplify the notations, we set
\begin{equation*}
\begin{aligned}
&c_i=\cosh r_i,\ s_i=\sinh r_i,\ \kappa_i=\coth r_i, \\
&\cos\Phi_{12}=x,\ \cos\Phi_{13}=y,\ \cos\Phi_{14}=z,\ \\
&\cos\Phi_{23}=u,\ \cos\Phi_{24}=v,\ \cos\Phi_{34}=w.
 \end{aligned}
\end{equation*}
See Figure \ref{Figure: tetrahedron} for the notations.
\begin{figure}[!ht]
\centering
\begin{overpic}[scale=1.2]{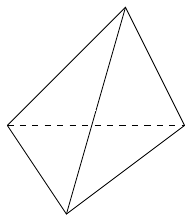}
\put (-1,41) {$2$}
\put (55,98) {$1$}
\put (28,-3) {$3$}
\put (84,41) {$4$}
\put (11,19) {$u$}
\put (28,45) {$v$}
\put (58,19) {$w$}
\put (25,71) {$x$}
\put (48,58) {$y$}
\put (72,70) {$z$}
\end{overpic}
\caption{A hyperbolic tetrahedron.}
\label{Figure: tetrahedron}
\end{figure}

Submitting (\ref{Eq: Thurston hyperbolic length}) into (\ref{Eq: G}) yields
\begin{equation*}
\begin{aligned}
\det G
%=&\left|
% \begin{array}{ccccc}
%  1 & \cosh l_{12} & \cosh l_{13} & \cosh l_{14}\\
%  \cosh l_{12} & 1 & \cosh l_{23} & \cosh l_{24}\\
%   \cosh l_{13} & \cosh l_{23} & 1 & \cosh l_{34}\\
%    \cosh l_{14} & \cosh l_{24} & \cosh l_{34} & 1 \\
%   \end{array}
%   \right|\\
=&\ s_1^2s_2^2s_3^2s_4^2(x^2+y^2+z^2+u^2+v^2+w^2 +x^2w^2+y^2v^2+z^2u^2\\
&-2xyvw-2xzuw-2yzuv-3)\\
&+s_1^2s_2^2s_3^2(x^2+y^2+u^2+2xyu-1)
+s_1^2s_2^2s_4^2 (x^2+z^2+v^2+2xzv-1)\\
&+s_1^2s_3^2s_4^2(y^2+z^2+w^2+2yzw-1)
+s_2^2s_3^2s_4^2 (u^2+v^2+w^2+2uvw-1)\\
&-2c_1c_2s_1s_2s_3^2s_4^2\left[x(1-w^2)+yvw+zuw+yu+zv\right]\\
&-2c_1c_3s_1s_2^2s_3s_4^2\left[y(1-v^2)+xvw+zuv+xu+zw\right]\\
&-2c_1c_4s_1s_2^2s_3^2s_4\left[z(1-u^2)+xuw+yuv+xv+yw\right]\\
&-2c_2c_3s_1^2s_2s_3s_4^2\left[u(1-z^2)+xzw+yzv+xy+vw\right]\\
&-2c_2c_4s_1^2s_2s_3^2s_4\left[v(1-y^2)+xyw+yzu+xz+uw\right]\\
&-2c_3c_4s_1^2s_2^2s_3s_4\left[w(1-x^2)+xyv+xzu+yz+uv\right].
\end{aligned}
\end{equation*}
Then
\begin{equation*}
\det G=s_1^2s_2^2s_3^2s_4^2 Q,
\end{equation*}
where
\begin{equation}\label{Eq: Q}
\begin{aligned}
Q=&\ Q_1+A_4(\kappa_4^2-1)+A_3(\kappa_3^2-1)
+A_2(\kappa_2^2-1)+A_1(\kappa_1^2-1)\\
&-2\kappa_1\kappa_2\left[x(1-w^2)+yvw+zuw+yu+zv\right]\\
&-2\kappa_1\kappa_3\left[y(1-v^2)+xvw+zuv+xu+zw\right]\\
&-2\kappa_1\kappa_4\left[z(1-u^2)+ xuw+yuv+xv+yw\right]\\
&-2\kappa_2\kappa_3\left[u(1-z^2)+ xzw+yzv+xy+vw\right]\\
&-2\kappa_2\kappa_4\left[v(1-y^2)+ xyw+yzu+xz+uw\right]\\
&-2\kappa_3\kappa_4\left[w(1-x^2)+xyv+xzu+yz+uv\right],
\end{aligned}
\end{equation}
with
\begin{equation}\label{Eq: Q_1}
\begin{aligned}
Q_1=&\ x^2+y^2+z^2+u^2+v^2+w^2+x^2w^2+y^2v^2 \\
&+z^2u^2-2xyvw-2xzuw-2yzuv-3,
\end{aligned}
\end{equation}
and
\begin{equation}\label{Eq: A}
\begin{aligned}
A_4=x^2+y^2+u^2+2xyu-1,\ A_3=x^2+z^2+v^2+2xzv-1,\\
A_2=y^2+z^2+w^2+2yzw-1,\ A_1=u^2+v^2+w^2+2uvw-1.
\end{aligned}
\end{equation}

As an application of Proposition \ref{Prop: non-degenerate condition}, we have the following result.

\begin{lemma}\label{Lem: non-degenerate}
A hyperbolic tetrahedron $\sigma=\{1234\}$ with edge lengths defined by (\ref{Eq: Thurston hyperbolic length}) is non-degenerate if and only if $Q<0$.
\end{lemma}

We futher have the following interesting observation.
\begin{lemma}\label{Lem: trivial}
For a hyperbolic tetrahedron $\sigma=\{1234\}$ with edge lengths defined by (\ref{Eq: Thurston hyperbolic length}),
if $A_1\leq 0, A_2\leq 0, A_3\leq 0, A_4\leq 0$,
then $Q\leq 0$.
Furthermore, $Q=0$ if and only if $(x,y,z,u,v,w)\in \widetilde{S}$, where 
\begin{eqnarray}\label{Eq: exception weights}
\widetilde{S}=\{(x,y,z,u,v,w)\in [0,1]^6 \mid (1,0,0,0,0,1), (0,1,0,0,1,0), (0,0,1,1,0,0)\}.
\end{eqnarray}
\end{lemma}
\proof
Given $A_i\leq 0$ for all $i\in\{1,2,3,4\}$ and $\kappa_i=\coth r_i>1$, 
(\ref{Eq: Q}) yields $Q\leq Q_1$, while (\ref{Eq: A}) gives
\begin{equation}\label{Eq: F5}
x^2+y^2+z^2+u^2+v^2+w^2\leq 2-xyu-xzv-yzw-uvw.
\end{equation}
Submitting (\ref{Eq: F5}) into (\ref{Eq: Q_1}) yields
\begin{equation*}
\begin{aligned}
Q_1\leq&\ 2-xyu-xzv-yzw-uvw +x^2w^2+y^2v^2+z^2u^2\\
&\ \ \ \ -2xyvw-2xzuw-2yzuv-3\\
=&-1-xyu-xzv-yzw-uvw-xyvw-xzuw-yzuv\\
&+\frac{1}{2}[(xw-yv)^2+(xw-zu)^2+(yv-zu)^2]\\
\leq & -1+\frac{1}{2}[(xw-yv)^2+(xw-zu)^2+(yv-zu)^2].
 \end{aligned}
\end{equation*}
For the term $(xw-yv)^2+(xw-zu)^2+(yv-zu)^2$, without of loss of generality, we assume
\begin{equation*}
0\leq xw\leq yv\leq zu\leq 1.
\end{equation*}
Under this condition, we have
\begin{equation*}
(xw-yv)^2+(xw-zu)^2+(yv-zu)^2\leq 2(xw-zu)^2,
\end{equation*}
which implies
\begin{equation*}
Q_1\leq  -1+\frac{1}{2}[(xw-yv)^2+(xw-zu)^2+(yv-zu)^2]
\leq -1+(xw-zu)^2\leq 0,
\end{equation*}
with equality achieved if and only if $zu=1$, $xw=0$,  and $yv=0$ or $1$.
Hence $Q\leq 0$.

Note that $zu=1$ implies $z=u=1$.
Submitting $z=u=1$ into the condition $A_i\leq0$ for all $i\in\{1,2,3,4\}$, we have
\begin{equation*}
\begin{aligned}
&x^2+y^2+2xy\leq 0,\ x^2+v^2+2xv\leq 0,\\
&y^2+w^2+2yw\leq 0,\ v^2+w^2+2vw\leq 0,
\end{aligned}
\end{equation*}
which implies that $x=y=v=w=0$. 
Substituting $z=u=1$ and $x=y=v=w=0$ into (\ref{Eq: Q}), we obtain $Q=0$.
By symmetry, $Q=0$ also holds in the cases that $x=w=1, y=z=u=v=0$ and $y=v=1, x=z=u=w=0$.
\qed

Please refer to Figure \ref{Figure: four circles} for the case that $x=w=1$ and $y=z=u=v=0$.
\begin{figure}[!ht]
\centering
\begin{overpic}[scale=0.5]{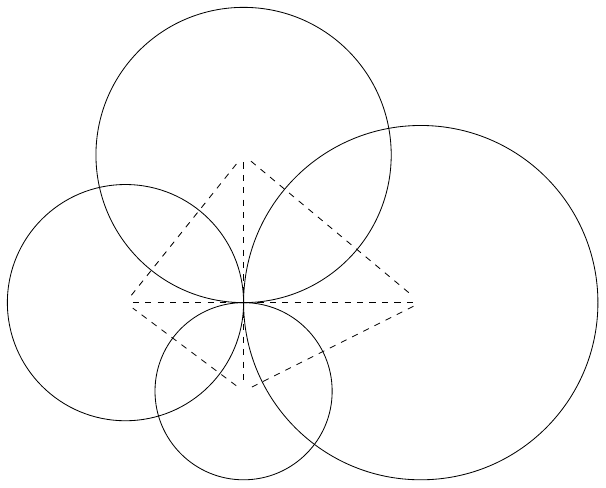}
\put (16,24) {$1$}
\put (37,56) {$4$}
\put (37,8) {$3$}
\put (72,24) {$2$}
\end{overpic}
\caption{Degenerate sphere packing metric.}
\label{Figure: four circles}
\end{figure}
In the following, we require the weight on the edges to be a map $\eta: E_{\sigma}\rightarrow [0,1]$ with $\eta(E_\sigma)\in [0, 1]^6\setminus \widetilde{S}$, where $\eta_{ij}=\cos\Phi_{ij}$.

Set
\begin{equation*}
\begin{aligned}
Q_2=&\ Q_1-A_1-A_2-A_3-A_4\\
=&\ 1+x^2w^2+y^2v^2+z^2u^2-(x^2+y^2+z^2+u^2+v^2+w^2)\\
&-2xyvw-2xzuw-2yzuv-2xyu-2xzv-2yzw-2uvw.
\end{aligned}
\end{equation*}

\begin{lemma}\label{Lem: Q_2<0}
If there exists an index $i\in \{1,2,3,4\}$ such that $A_i\geq0$, then $Q_2\leq0$.
In particular, if $A_i>0$, then $Q_2<0$.
\end{lemma}
\proof
%Note that $x,y,z,u,v,w\in [0,1]$ and $Q_2$ is decreasing in $x,y,z,u,v,w$.
If $A_1=u^2+v^2+w^2+2uvw-1\geq0$, then
\begin{equation*}
\begin{aligned}
Q_2=&\ (1-u^2-v^2-w^2-2uvw)+x^2(w^2-1)+y^2(v^2-1)+z^2(u^2-1)\\
&-2xyvw-2xzuw-2yzuv-2xyu-2xzv-2yzw\leq0.
\end{aligned}
\end{equation*}
Specially, if $A_1>0$, then $Q_2<0$.
Similar results also apply to $A_2$, $A_3$, and $A_4$.
\qed

By Lemma \ref{Lem: non-degenerate}, $r=(r_1, r_2,r_3, r_4)\in \mathbb{R}^4_{>0}$ is a degenerate Thurston's hyperbolic sphere packing metric if and only if $\kappa=(\kappa_1,\kappa_2,\kappa_3,\kappa_4)=(\coth r_1,\coth r_2,\coth r_3,\coth r_4)\in (1,+\infty)^4$ satisfies
\begin{equation*}
Q\geq 0.
\end{equation*}
This is equivalent to 
\begin{eqnarray*}
A_1\kappa_1^2+B_1\kappa_1+C_1\geq 0
\end{eqnarray*}
with
\begin{equation}\label{Eq: ABC}
\begin{aligned}
B_1=&-2[\kappa_2(x(1-w^2)+yvw+zuw+yu+zv)\\
&+\kappa_3(y(1-v^2)+xvw+zuv+xu+zw)\\
&+\kappa_4(z(1-u^2)+xuw+yuv+xv+yw)],\\
C_1=&\ Q_2+A_2\kappa_2^2+A_3\kappa_3^2+A_4\kappa_4^2\\
&-2\kappa_2\kappa_3\left[u(1-z^2)+ xzw+yzv+xy+vw\right]\\
&-2\kappa_2\kappa_4\left[v(1-y^2)+ xyw+yzu+xz+uw\right]\\
&-2\kappa_3\kappa_4\left[w(1-x^2)+xyv+xzu+yz+uv\right],
\end{aligned}
\end{equation}

\begin{corollary}\label{Cor: positive Delta}
If $A_1=u^2+v^2+w^2+2uvw-1>0$, then $\Delta_1=B_1^2-4A_1C_1>0$.
\end{corollary}
\proof
Lengthy and direct calculations yield
\begin{equation*}
\begin{aligned}
\Delta_1=-4Q_2\bigg[&A_1
+\kappa_2^2(1-w^2)+\kappa_3^2(1-v^2)+\kappa_4^2(1-u^2)\\
&+2\kappa_2\kappa_3(u+vw)+2\kappa_2\kappa_4(v+uw)
+2\kappa_3\kappa_4(w+uv)\bigg].
\end{aligned}
\end{equation*}
By Lemma \ref{Lem: Q_2<0}, if $A_1>0$, then $Q_2<0$ and thus $\Delta_1>0$.
\qed

\begin{remark}
If we take $Q$ as a quadratic function of $\kappa_i$, $i\in \{2,3,4\}$, then results similar to Corollary \ref{Cor: positive Delta} hold for $i\in \{2,3,4\}$.
\end{remark}

Set
\begin{equation}\label{Eq: h_i}
\begin{aligned}
h_1=-A_1\kappa_1&+\kappa_2[x(1-w^2)+yvw+zuw+yu+zv]\\
    &+\kappa_3[y(1-v^2)+xvw+zuv+xu+zw]\\
    &+\kappa_4[z(1-u^2)+xuw+yuv+xv+yw],\\
h_2=-A_2\kappa_2&+\kappa_1[x(1-w^2)+yvw+zuw+yu+zv]\\
    &+\kappa_3[u(1-z^2)+xzw+yzv+xy+vw]\\
    &+\kappa_4[v(1-y^2)+xyw+yzu+xz+uw],\\
h_3=-A_3\kappa_3&+\kappa_1[y(1-v^2)+xvw+zuv+xu+zw]\\
    &+\kappa_2[u(1-z^2)+xzw+yzv+xy+vw]\\
    &+\kappa_4[w(1-x^2)+xyv+xzu+yz+uv],\\
h_4=-A_4\kappa_4&+\kappa_1[z(1-u^2)+xuw+yuv+xv+yw]\\
    &+\kappa_2[v(1-y^2)+xyw+yzu+xz+uw]\\
    &+\kappa_3[w(1-x^2)+xyv+xzu+yz+uv].
\end{aligned}
\end{equation}
Then 
\begin{equation}\label{Eq: F4}
Q=Q_2-(\kappa_1h_1+\kappa_2h_2+\kappa_3h_3+\kappa_4h_4).
\end{equation}

Suppose $r=(r_1, r_2, r_3, r_4)\in \mathbb{R}^4_{>0}$ is a degenerate Thurston hyperbolic sphere packing metric.
Then by Lemma \ref{Lem: non-degenerate}, $Q\geq 0$.
Note that the weight $\eta(E_\sigma)\in [0, 1]^6\setminus \widetilde{S}$.
By Lemma \ref{Lem: trivial}, at least one of $A_1, A_2, A_3, A_4$ is positive.
Without loss of generality, assume $A_1>0$.
By Lemma \ref{Lem: Q_2<0}, $Q_2<0$.
From (\ref{Eq: F4}), at least one of $h_1, h_2, h_3, h_4$ is negative.

\begin{lemma}\label{Lem: none zero}
Suppose $(r_1, r_2, r_3, r_4)\in \mathbb{R}^4_{>0}$ is a degenerate Thurston's hyperbolic sphere packing metric on the tetrahedron $\sigma=\{1234\}$ with the weight $\eta(E_\sigma)\in [0,1]^6\setminus \widetilde{S}$.
Then there exists no subset $\{i,j\}\subset \{1,2,3,4\}$ such that $h_i\leq0$ and $h_j\leq0$.
\end{lemma}
\proof
Without loss of generality, we assume $h_1\leq0$ and $h_2\leq0$.
By (\ref{Eq: h_i}), we have
\begin{equation}\label{Eq: F1}
\begin{aligned}
A_1\kappa_1
\geq&\ \kappa_2[x(1-w^2)+yvw+zuw+yu+zv]\\
    &+\kappa_3[y(1-v^2)+xvw+zuv+xu+zw]\\
    &+\kappa_4[z(1-u^2)+xuw+yuv+xv+yw],
\end{aligned}
\end{equation}
and 
\begin{equation}\label{Eq: F2}
\begin{aligned}
A_2\kappa_2
\geq&\ \kappa_1[x(1-w^2)+yvw+zuw+yu+zv]\\
    &+\kappa_3[u(1-z^2)+xzw+yzv+xy+vw]\\
    &+\kappa_4[v(1-y^2)+xyw+yzu+xz+uw].
\end{aligned}
\end{equation}
Hence, $A_1\geq 0$ and $A_2\geq 0$.
Moreover, multiplying the lefthand side of (\ref{Eq: F1}) with the lefthand side of (\ref{Eq: F2}) gives $A_1A_2\kappa_1\kappa_2$.
Multiplying the righthand side of (\ref{Eq: F1}) with the righthand side of (\ref{Eq: F2}) gives the sum of $\kappa_1\kappa_2[x(1-w^2)+yu+zv+yvw+zuw]^2$ and some nonnegative terms.
Thus
\begin{equation*}
Q_3:=A_1A_2-[x(1-w^2)+yvw+zuw+yu+zv]^2\geq0.
\end{equation*}
However, direct calculations give $Q_3=(1-w^2)Q_2$.
Lemma \ref{Lem: Q_2<0} shows that $Q_2\leq0$, and then $Q_3\leq0$.
Therefore, $Q_3\equiv0$, which implies $w=1$ or $Q_2=0$.

In the case of $w=1$,
since $Q_3\equiv0$,  the inequality (\ref{Eq: F1}) should be an equality, and the second and third terms in the righthand side of (\ref{Eq: F1}) vanish.
This implies $y=z=0$ and $x(u+v)=0$.
For (\ref{Eq: F2}), similar arguments yield $u=v=0$ and $x(y+z)=0$.
From the condition that the weight $\eta(E_\sigma)\in [0,1]^6\setminus \widetilde{S}$, we have $x\neq 1$.
Substituting $y=z=u=v=0$, $w=1$, and $x\neq1$ into (\ref{Eq: Q}) gives
\begin{equation*}
Q=(x^2-1)(\kappa_3+\kappa_4)^2<0.
\end{equation*}
By Lemma \ref{Lem: non-degenerate}, this contradicts the condition that $r$ is a degenerate Thurston's hyperbolic sphere packing metric.

In the case of $Q_2=0$, Lemma \ref{Lem: Q_2<0} gives $A_i\leq 0$ for all $i\in\{1,2,3,4\}$.
By Lemma \ref{Lem: trivial}, this implies $Q\leq 0$.
From the condition $\eta(E_\sigma)\in [0,1]^6\setminus \widetilde{S}$, Lemma \ref{Lem: trivial} further yields $Q<0$.
By Lemma \ref{Lem: non-degenerate}, this contradicts the requirement that $r$ is a degenerate Thurston's hyperbolic sphere packing metric.
\qed

\begin{corollary}\label{Cor: one negative three positive}
Suppose $(r_1, r_2, r_3, r_4)\in \mathbb{R}^4_{>0}$ is
a degenerate Thurston's hyperbolic sphere packing metric on a tetrahedron with $\eta(E_\sigma)\in [0,1]^6\setminus \widetilde{S}$.
Then one of $h_1, h_2, h_3, h_4$ is negative and the others are positive.
\end{corollary}

\begin{theorem}\label{Thm: simply connect}
Suppose $\sigma=\{1234\}$ is a hyperbolic tetrahedron of $(M, \mathcal{T})$. 
The admissible space $\Omega_{1234}(\eta)$ of Thurston's hyperbolic sphere packing metrics for a given weight $\eta$ with $\eta(E_\sigma)\in [0,1]^6\setminus \widetilde{S}$ is a simply connected nonempty set.
\end{theorem}
\proof
Suppose $r=(r_1, r_2, r_3, r_4)$ is a degenerate Thurston's hyperbolic sphere packing metric.
Then by Corollary \ref{Cor: one negative three positive},
one of $h_1, h_2, h_3, h_4$ is negative and the others are positive.
Without loss of generality, assume $h_1<0$.
By (\ref{Eq: h_i}),
\begin{equation*}
\begin{aligned}
\kappa_1A_1>&\ \kappa_2[x(1-w^2)+yvw+zuw+yu+zv]\\
    &+\kappa_3[y(1-v^2)+xvw+zuv+xu+zw]\\
    &+\kappa_4[z(1-u^2)+xuw+yuv+xv+yw]\geq 0,
\end{aligned}
\end{equation*}
which implies that
$A_1=u^2+v^2+w^2+2uvw-1>0$.
In this case, take $Q\geq 0$ as a quadratic inequality $A_1\kappa_1^2+B_1\kappa_1+C_1\geq 0$ of $\kappa_1$,
where $B_1, C_1$ are given by (\ref{Eq: ABC}).
By Corollary \ref{Cor: positive Delta}, we have $\Delta_1=B_1^2-4A_1C_1>0$.
Therefore,
\begin{eqnarray*}
\kappa_1\geq\frac{-B_1+\sqrt{\Delta_1}}{2A_1}   \quad \text{or}  \quad
\kappa_1\leq\frac{-B_1-\sqrt{\Delta_1}}{2A_1}.
\end{eqnarray*}
The definition of $h_i$ in (\ref{Eq: h_i}) can imply $h_i=-\frac{1}{2}(2A_i\kappa_i+B_i)$.
Then $h_1<0$ is equivalent to $\kappa_1>-\frac{B_1}{2A_1}$.
Therefore,
$\kappa_1\geq\frac{-B_1+\sqrt{\Delta_1}}{2A_1}$,
which also implies $\kappa_1>\frac{-B_1}{2A_1}$ and $h_1<0$.

In the case where $A_1>0$, set
\begin{eqnarray*}
U_1=\{(r_1, r_2, r_3, r_4)\in \mathbb{R}^4_{>0}\mid \kappa_1\geq \frac{-B_1+\sqrt{\Delta_1}}{2A_1}\},
\end{eqnarray*}
which is a closed subset of $\mathbb{R}^4_{>0}$ bounded by an analytic graph over $\mathbb{R}^3_{>0}$.
Similarly, we define $U_2, U_3, U_4$ whenever the corresponding $A_i>0$.
Note that $U_1,U_2,U_3,U_4$ are mutually disjoint by Lemma \ref{Lem: none zero}, if they are nonempty.

If $A_\mu>0$ for $\mu\in P\subset \{1,2,3,4\}$ and $A_\nu\leq 0$ for $\nu\in \{1,2,3,4\}\setminus P$, then the space of degenerate Thurston's hyperbolic sphere packing metrics
is $\cup_{\mu\in P}U_\mu$.
The corresponding admissible space of Thurston's sphere packing metric is
\begin{equation*}
\begin{aligned}
\Omega_{1234}(\eta)=\mathbb{R}^4_{>0}\setminus\cup_{\mu\in P}U_\mu,
\end{aligned}
\end{equation*}
which is nonempty.
Obviously, $\Omega_{1234}(\eta)$ is homotopy equivalent to $\mathbb{R}^4_{>0}$, because
$U_1, U_2, U_3, U_4$ are mutually disjoint and bounded by analytic graphs on $\mathbb{R}^3_{>0}$.
This implies that, for any fixed weight $\eta$ with $\eta(E_\sigma)\in [0, 1]^6\setminus \widetilde{S}$,
the admissible space of Thurston's hyperbolic sphere packing metrics $\Omega_{1234}(\eta)$ is simply connected.  
\qed

Theorem \ref{Thm: simply connect} gives the following corollary,
which is another representation of Lemma \ref{Lem: trivial}.

\begin{corollary}\label{Cor: AS R4}
Suppose $\sigma=\{1234\}$ is a hyperbolic tetrahedron of $(M, \mathcal{T})$ with the weight $\eta$ on $E_\sigma$ satisfying $\eta(E_\sigma)\in [0,1]^6\setminus \widetilde{S}$.
Then the admissible space $\Omega_{1234}(\eta)=\mathbb{R}^4_{>0}$ if and only if
$A_1\leq 0, A_2\leq 0, A_3\leq 0, A_4\leq 0$.
Specially, if
\begin{equation}\label{Eq: sum of Phi > pi}
\Phi_{ij}+\Phi_{ik}+\Phi_{jk}\geq \pi
\end{equation}
for any face $\{ijk\}\in F_\sigma$, then the admissible space $\Omega_{1234}(\eta)$ of Thurston's hyperbolic sphere packing metrics is $\mathbb{R}^{4}_{>0}$.
\end{corollary}
\proof
The first part of Corollary \ref{Cor: AS R4} follows from the proof of Theorem \ref{Thm: simply connect}.
For the second part, note that
\begin{equation}\label{eq A identity}
\begin{aligned}
A_1
=&\ \cos^2\Phi_{23}+\cos^2\Phi_{24}+\cos^2\Phi_{34}+2\cos\Phi_{23}\cos\Phi_{24}\cos\Phi_{34}-1\\
=&\ (\cos\Phi_{23}+\cos\Phi_{24}\cos\Phi_{34})^2-\sin^2\Phi_{24}\sin^2\Phi_{34}\\
=&\ 4\cos\frac{\Phi_{23}+\Phi_{24}+\Phi_{34}}{2}\cos\frac{\Phi_{24}+\Phi_{34}-\Phi_{23}}{2}
\cos\frac{\Phi_{23}-\Phi_{24}+\Phi_{34}}{2}\cos\frac{\Phi_{23}+\Phi_{24}-\Phi_{34}}{2}.
\end{aligned}
\end{equation}
Since $\Phi_{ij}\in [0, \frac{\pi}{2}]$, we have
\begin{equation*}
\begin{aligned}
\cos\frac{\Phi_{24}+\Phi_{34}-\Phi_{23}}{2}\geq 0,\quad
\cos\frac{\Phi_{23}-\Phi_{24}+\Phi_{34}}{2}\geq 0,\quad
\cos\frac{\Phi_{23}+\Phi_{24}-\Phi_{34}}{2}\geq 0.
\end{aligned}
\end{equation*}
Therefore, if $\Phi_{23}+\Phi_{24}+\Phi_{34}\geq  \pi,$ then
$\cos\frac{\Phi_{23}+\Phi_{24}+\Phi_{34}}{2}\leq 0$ and $A_1\leq 0$.
Similarly, if
$$\Phi_{13}+\Phi_{14}+\Phi_{34}\geq \pi,\quad \Phi_{12}+\Phi_{14}+\Phi_{24}\geq \pi, \quad \Phi_{12}+\Phi_{13}+\Phi_{23}\geq \pi,$$
then $A_2\leq 0, A_3\leq0, A_4\leq0$ respectively. \qed

\begin{remark}
Condition (\ref{Eq: sum of Phi > pi}) in Corollary \ref{Cor: AS R4} includes the case $\Phi_{ij}\in [\frac{\pi}{3}, \frac{\pi}{2}]$.
Furthermore, (\ref{Eq: sum of Phi > pi}) permits some intersection angles to lie in $(0, \frac{\pi}{3})$.
Nevertheless, the configurations illustrated in Figure \ref{Figure: four circles} must be excluded.
Similar phenomena arise for circle packings with obtuse intersection angles \cite{Z}
and inversive distance circle packings with negative inversive distance \cite{X2} on surfaces.
\end{remark}

\begin{remark}
The simply connectedness of the admissible space of tangential sphere packing metrics for a single tetrahedron was first proved by Cooper-Rivin \cite{CR}.
Xu \cite{X1} gave a different proof of the simply connectedness with an explicit description of the boundary of the admissible space.
Ge-Hua \cite{GH} gave some new viewpoints on the method in \cite{X1} for tangential sphere packings.
The proof of Theorem \ref{Thm: simply connect} for Thurston's hyperbolic sphere packing metrics unifies the proofs for simply connectedness in \cite{GH,X1}.
This method has been applied to give a new proof of the Bowers-Stephenson's conjecture for inversive distance circle packings on triangulated surfaces \cite{X3}, 
which simplifies the proofs of the Bowers-Stephenson's conjecture in \cite{Guo,L2,X2}.
\end{remark}

For any weight $\eta$ with $\eta(E_\sigma)\in W=[0,1]^6\setminus \widetilde{S}$, 
by Theorem \ref{Thm: simply connect},
the admissible space $\Omega_{1234}(\eta)$ is
a 4-dimensional simply connected nonempty open subset of $\mathbb{R}^4_{>0}$. 
Using $\Omega_{1234}(\eta)$, we can further define the following 10-dimensional. Set
\begin{equation*}
\begin{aligned}
\Omega_{1234}=\cup_{\eta\in W}\, \Omega_{1234}(\eta)=\{(\eta, r)\in W\times \mathbb{R}^4_{>0}\mid Q(\eta, r)<0\},
\end{aligned}
\end{equation*}
where we take $Q$ as a function of $\eta$ and $r$.
To simplify the notations, here we use $\eta\in W$ to denote $\eta(E_\sigma)\in W$.

\begin{corollary}\label{Cor: connectness of 10-dim admissible space}
The set $\Omega_{1234}$ is connected.
\end{corollary}
\proof
Suppose $\eta_0\in W$.
By the nonemptiness of $\Omega_{1234}(\eta_0)$, there exists $r_0\in \Omega_{1234}(\eta_0)$ such that $Q(\eta_0, r_0)<0$.
Consider the continuous function $Q(\eta, r_0)$ in $\eta$.
Since $Q(\eta_0, r_0)<0$, there exists a connected neighborhood $U_{\eta_0}\subset W$ of $\eta_0$ satisfying $Q(\eta, r_0)<0$ for all $\eta\in U_{\eta_0}$.
By Theorem \ref{Thm: simply connect}, this implies that for every $\eta\in U_{\eta_0}$, any two points $(\eta, r_A)\in\Omega_{1234}$ and $(\eta_0, r_B)\in\Omega_{1234}$ can be joined by a path in $\Omega_{1234}$.
In this sense, the fibers $\Omega_{1234}(\eta)$ for $\eta\in U_{\eta_0}$ and $\Omega_{1234}(\eta_0)$ are path-connected within $U_{\eta_0}$.
Treating each fiber $\Omega_{1234}(\eta)$ as a single point,
we deduce that for any $\eta_A, \eta_B\in W=[0,1]^6\setminus \widetilde{S}$, there exists a path from $\Omega_{1234}(\eta_A)$ to $\Omega_{1234}(\eta_B)$ in $W$.
This follows from the connectedness of $W$ and the finite covering theorem, and it implies that $\Omega_{1234}$ is connected.
\qed

\section{Variations for tetrahedra generated by Thurston's sphere packings}\label{section 3}

\subsection{Variations of dihedral angles}

Let $\sigma=\{1234\}$ be the vertex set of a tetrahedron.
For $i\in\sigma$, set
\begin{equation*}
s_i = \sinh r_i,\quad 
\kappa_i = \coth r_i > 1,\quad 
\eta_{ij} = \cos\Phi_{ij},\quad 
\kappa = (\kappa_1,\dots,\kappa_4)^T.
\end{equation*}
Then (\ref{Eq: Thurston hyperbolic length}) is rewritten as
\begin{equation}\label{eq:length-law}
 \cosh l_{ij}
 =\cosh r_i\cosh r_j+\eta_{ij}\sinh r_i\sinh r_j
 =s_is_j(\kappa_i\kappa_j+\eta_{ij}).
\end{equation}
Define the symmetric matrices
\begin{equation*}
P=\left(
   \begin{array}{cccc}
    -1 & \eta_{ij} & \eta_{ik} & \eta_{il} \\
     \eta_{ij} & -1 & \eta_{jk} & \eta_{jl} \\
     \eta_{ik} & \eta_{jk} & -1 & \eta_{kl} \\
     \eta_{il} & \eta_{jl} & \eta_{kl} &-1
    \end{array}
 \right),\quad 
M = P + \kappa\kappa^T.
\end{equation*}
Set $S=\diag(s_1,s_2,s_3,s_4)$.
The matrix $G$ in (\ref{Eq: G}) is rewritten as 
\begin{equation}\label{th-congruence-G}
G = S MS.
\end{equation}
By Proposition \ref{Prop: non-degenerate condition},
for a non-degenerate hyperbolic tetrahedron, 
all $3\times 3$ principal minors are positive and the $4\times 4$ determinant is negative. 
Therefore,
\begin{equation}\label{Eq: F26}
\det M=\frac{\det G}{\prod_i s_i^2}<0,
\qquad
\det M_{\widehat k}
=\frac{\det G_{\widehat k}}{\prod_{i\ne k}s_i^2}>0.
\end{equation}
Here $\widehat{k}$ denotes deletion of the $k$-th index, and submatrices and vectors retain the original vertex labeling.
Let $M_{ij}$ denote the submatrix of the square matrix $M$ formed by the $(i, i), (i, j), (j, i)$ and $(j, j)$ entries.

Define
\begin{equation}\label{th-inverse}
H=M^{-1},\qquad h=H\kappa,\qquad q=-\det M>0.
\end{equation}
By (\ref{Eq: F26}), the adjugate formula gives
\begin{equation*}
H_{kk}=\frac{\det M_{\widehat k}}{\det M}<0.
\end{equation*}
For each edge, set
\begin{equation}\label{th-E}
\begin{aligned}
E_{ij}:=-\det M_{\{i,j\}}
&=(\kappa_i\kappa_j+\eta_{ij})^2
 -(\kappa_i^2-1)(\kappa_j^2-1)\\
&=\kappa_i^2+\kappa_j^2+2\eta_{ij}\kappa_i\kappa_j
  +\eta_{ij}^2-1 \\
&=\frac{\sinh^2l_{ij}}{s_i^2s_j^2}>0,
\end{aligned}
\end{equation}
where (\ref{eq:length-law}) and the fact $s_i^2(\kappa_i^2-1)=1$ are used in the third line.
For distinct $i,j,k,l$, define
\begin{equation}\label{th-b-definition}
F^{(k)}=M_{\{i,j,l\}}^{-1}=M_{\widehat k}^{-1},\quad
b_l^{(k)}=(F^{(k)}\kappa_{\widehat k})_l=(M_{\widehat k}^{-1}\kappa_{\widehat k})_l.
\end{equation}

\begin{lemma}
\begin{equation}\label{th-b-block}
b_l^{(k)}=h_l-\frac{H_{lk}}{H_{kk}}h_k.
\end{equation}
\end{lemma}
\proof
Set
\[
M=\begin{pmatrix}
M_{\widehat k}&M_{\widehat k,k}\\
M_{k,\widehat k}&M_{kk}
\end{pmatrix}.
\]
By (\ref{th-inverse}), $Mh=\kappa$ gives
\begin{equation}\label{Eq: F27}
M_{\widehat k}h_{\widehat k}+M_{\widehat k,k}h_k=\kappa_{\widehat k}.
\end{equation}
The block elimination yields
\[
\begin{pmatrix}I&0\\-M_{k,\widehat k}F^{(k)} &1\end{pmatrix}M
=\begin{pmatrix}
M_{\widehat k}&M_{\widehat k,k}\\
0&M_{kk}-M_{k,\widehat k}F^{(k)} M_{\widehat k,k}
\end{pmatrix}.
\]
It follows that
\[
\det M=\det M_{\widehat k}
\bigl(M_{kk}-M_{k,\widehat k}F^{(k)} M_{\widehat k,k}\bigr).
\]
Note that both determinants are nonzero.
%, and hence the Schur complement is nonzero. 
Inverting the triangular factors gives
\begin{equation}\label{Eq: F24}
H
=\begin{pmatrix}
H_{\widehat k}&H_{\widehat k,k}\\
H_{k,\widehat k}&H_{kk}
\end{pmatrix}
=\begin{pmatrix}
F^{(k)}+F^{(k)} M_{\widehat k,k}H_{kk}M_{k,\widehat k}F^{(k)}
&-F^{(k)} M_{\widehat k,k}H_{kk}\\
-H_{kk}M_{k,\widehat k}F^{(k)}&(M_{kk}-M_{k,\widehat k}F^{(k)} M_{\widehat k,k})^{-1}
\end{pmatrix}.
\end{equation}
Using $H_{\widehat k,k}=-F^{(k)} M_{\widehat k,k}H_{kk}$ and (\ref{th-b-definition}), and 
multiplying both sides of (\ref{Eq: F27}) by $F^{(k)}$, we obtain
\begin{equation*}
F^{(k)}\kappa_{\widehat k}
=h_{\widehat k}+F^{(k)} M_{\widehat k,k}h_k
=h_{\widehat k}-\frac{H_{\widehat k,k}}{H_{kk}}h_k.
\end{equation*}
Taking the $l$-th component proves (\ref{th-b-block}).
\qed

Set $R=\diag(s_1^2,s_2^2,s_3^2,s_4^2)$.
Take an arbitrary vector $v\in\mathbb R^4$ and normalize the radius variation by
\begin{equation}\label{Eq: F31}
 \delta r=Rv, \qquad \delta r_i=s_i^2v_i.
\end{equation}
The quotient rule gives
\[
\frac{d\kappa_i}{dr_i}
=\frac{\sinh^2r_i-\cosh^2r_i}{\sinh^2r_i}
=-s_i^{-2},
\qquad
\delta\kappa_i=-s_i^{-2}\delta r_i=-v_i.
\]
Thus $\delta\kappa=-v$. 

We have the following basic variational formula for the dihedral angles in a hyperbolic tetrahedron generated by Thurston's hyperbolic sphere packings.
\begin{lemma}
Let $\beta_{ij}:=\beta_{ij,kh}$ be the dihedral angle along the edge $\{ij\}\in E$ in a hyperbolic tetrahedron $\{ijkh\}$ generated by Thurston's hyperbolic sphere packings.
Then
\begin{equation}\label{th-dbeta}
\delta\beta_{ij}
=-\frac{\sqrt q}{\sqrt{E_{ij}}}
\left[b_l^{(k)}(Hv)_k+b_k^{(l)}(Hv)_l\right].
\end{equation}
\end{lemma}
\proof
Let $\{k,l\}=\sigma\setminus\{i,j\}$. 
From \eqref{th-congruence-G}, $G^{-1}=S^{-1}HS^{-1}$.
It is also well known, see, e.g., \cite{MU}, that 
\begin{equation}\label{th-cos-beta}
\cos\beta_{ij}
=\frac{(G^{-1})_{kl}}{\sqrt{(G^{-1})_{kk}(G^{-1})_{ll}}}
=\frac{H_{kl}/(s_ks_l)}
 {\sqrt{(H_{kk}/s_k^2)(H_{ll}/s_l^2)}}
=\frac{H_{kl}}{\sqrt{H_{kk}H_{ll}}}.
\end{equation}
Differentiating \eqref{th-cos-beta} yields 
\begin{equation}\label{Eq: variation of dihedral angle 1}
\begin{aligned}
-\sin\beta_{ij}\,\delta\beta_{ij}
&=\frac{\delta H_{kl}}{\sqrt{H_{kk}H_{ll}}}
+H_{kl}\,\delta\bigl((H_{kk}H_{ll})^{-\frac{1}{2}}\bigr)\\
&=\frac{\delta H_{kl}}{\sqrt{H_{kk}H_{ll}}}
-\frac{H_{kl}}{2(H_{kk}H_{ll})^{\frac{3}{2}}}
\bigl(H_{ll}\delta H_{kk}+H_{kk}\delta H_{ll}\bigr)\\
&=\frac{1}{\sqrt{H_{kk}H_{ll}}}
\left[\delta H_{kl}-\frac{H_{kl}}2
\left(\frac{\delta H_{kk}}{H_{kk}}+\frac{\delta H_{ll}}{H_{ll}}\right)\right].
\end{aligned}
\end{equation}
Because $0<\beta_{ij}<\pi$, by (\ref{th-cos-beta}), we have 
\begin{equation}\label{Eq: sin beta}
\sin\beta_{ij}\sqrt{H_{kk}H_{ll}}
=(1-\cos^2\beta_{ij})\sqrt{H_{kk}H_{ll}}
=\sqrt{H_{kk}H_{ll}-H_{kl}^2}.
\end{equation}
Applying the Jacobi principal-minor identity, (\ref{th-inverse}) and (\ref{th-E}) gives
\begin{equation}\label{th-jacobi}
H_{kk}H_{ll}-H_{kl}^2
=\det H_{\{k,l\}}
=\frac{\det M_{\{i,j\}}}{\det M}
=\frac{E_{ij}}{q}>0.
\end{equation}
Substituting \eqref{Eq: sin beta} and \eqref{th-jacobi} into \eqref{Eq: variation of dihedral angle 1} gives
\begin{equation}\label{Eq: F33}
\delta\beta_{ij}
=-\frac{\sqrt q}{\sqrt{E_{ij}}}
\left[\delta H_{kl}-\frac{H_{kl}}2
\left(\frac{\delta H_{kk}}{H_{kk}}+\frac{\delta H_{ll}}{H_{ll}}\right)\right].
\end{equation}

Since $P$ is fixed in this variation, $\delta P=0$ and
\begin{equation}\label{Eq: F28}
\delta M
=\delta(P+\kappa\kappa^T)
=(\delta\kappa)\kappa^T+\kappa(\delta\kappa)^T
=-v\kappa^T-\kappa v^T.
\end{equation}
In components, this says
$\delta M_{ab}=-v_a\kappa_b-\kappa_av_b$ for all $a,b$,
including $\delta M_{aa}=-2\kappa_av_a$ on the diagonal.
Differentiating $MH=I$ gives 
\begin{equation}\label{Eq: F29}
0=\delta(MH)=(\delta M)H+M(\delta H).
\end{equation}
Left-multiplying both sides of (\ref{Eq: F29}) by $H$ and applying \eqref{Eq: F28} yields
\begin{equation*}
\delta H
=-H(\delta M)H
=Hv\kappa^TH+H\kappa v^TH
=(Hv)h^T+h(Hv)^T,
\end{equation*}
where $HM=I$ and $H\kappa=h$ are used. 
The corresponding entrywise computation is
\begin{align*}
\delta H_{ab}
&=-\sum_{p,t}H_{ap}(\delta M_{pt})H_{tb}\\
&=\sum_{p,t}H_{ap}(v_p\kappa_t+\kappa_pv_t)H_{tb}\\
&=\left(\sum_pH_{ap}v_p\right)\left(\sum_t\kappa_tH_{tb}\right)
 +\left(\sum_pH_{ap}\kappa_p\right)\left(\sum_tv_tH_{tb}\right)\\
&=(Hv)_ah_b+h_a(Hv)_b.
\end{align*}
In particular,
\begin{equation}\label{th-dH-components}
\delta H_{kl}=(Hv)_kh_l+h_k(Hv)_l,\quad
\delta H_{kk}=2h_k(Hv)_k,\quad
\delta H_{ll}=2h_l(Hv)_l.
\end{equation}
Substituting \eqref{th-dH-components} into \eqref{Eq: F33} gives
\begin{align*}
\delta\beta_{ij}
%&=-\frac{\sqrt q}{\sqrt{E_{ij}}}
%\left[\delta H_{kl}-\frac{H_{kl}}2
%\left(\frac{\delta H_{kk}}{H_{kk}}+\frac{\delta H_{ll}}{H_{ll}}\right)\right]\\
&=-\frac{\sqrt q}{\sqrt{E_{ij}}}
\left[(Hv)_kh_l+h_k(Hv)_l
-\frac{H_{kl}}2\left(\frac{2h_k(Hv)_k}{H_{kk}}
                    +\frac{2h_l(Hv)_l}{H_{ll}}\right)\right]\\
&=-\frac{\sqrt q}{\sqrt{E_{ij}}}
\left[\left(h_l-\frac{H_{lk}}{H_{kk}}h_k\right)(Hv)_k
+\left(h_k-\frac{H_{kl}}{H_{ll}}h_l\right)(Hv)_l\right].
\end{align*}
Using \eqref{th-b-block}, we obtain (\ref{th-dbeta}). 
\qed

\begin{remark}
The formula \eqref{th-dbeta} also gives the ordinary radius
derivatives. Indeed, $v_m=s_m^{-2}\delta r_m
=(\kappa_m^2-1)\delta r_m$, so
\begin{align*}
\delta\beta_{ij}
&=-\frac{\sqrt q}{\sqrt{E_{ij}}}
\sum_m\bigl(b_l^{(k)}H_{km}+b_k^{(l)}H_{lm}\bigr)v_m\\
&=-\frac{\sqrt q}{\sqrt{E_{ij}}}
\sum_m(\kappa_m^2-1)
\bigl(b_l^{(k)}H_{km}+b_k^{(l)}H_{lm}\bigr)\delta r_m.
\end{align*}
Thus, for every $m\in\sigma$,
\begin{equation*}
\frac{\partial\beta_{ij}}{\partial r_m}
=-\frac{\sqrt q(\kappa_m^2-1)}{\sqrt{E_{ij}}}
\bigl(b_l^{(k)}H_{km}+b_k^{(l)}H_{lm}\bigr).
\end{equation*}
\end{remark}

\subsection{Useful identities for the variation of dihedral angles}

Set
\begin{equation}\label{Eq: F30}
a_{ij}=\frac{\kappa_j+\eta_{ij}\kappa_i}{\kappa_i^2-1}>0.
\end{equation}
Let $D$ be the $4\times6$ radius--length Jacobian in the convention 
\begin{equation}\label{Eq: F32}
\delta l=D^T\delta r=D^TRv,
\end{equation}
where (\ref{Eq: F31}) is used.
Differentiating \eqref{eq:length-law} while keeping the intersection angles fixed gives
\begin{align*}
\sinh l_{ij}\frac{\partial l_{ij}}{\partial r_i}
=s_is_j(\kappa_j+\eta_{ij}\kappa_i),\quad
\sinh l_{ij}\frac{\partial l_{ij}}{\partial r_j}
=s_is_j(\kappa_i+\eta_{ij}\kappa_j).
\end{align*}
Dividing by \eqref{th-E} yields all nonzero entries of $D$:
\begin{equation}\label{th-D-entries}
D_{ij,i}
=\frac{\partial l_{ij}}{\partial r_i}
=\frac{\kappa_j+\eta_{ij}\kappa_i}{\sqrt{E_{ij}}}
=\frac{a_{ij}(\kappa_i^2-1)}{\sqrt{E_{ij}}},
\quad
D_{ij,j}
=\frac{\partial l_{ij}}{\partial r_j}
=\frac{\kappa_i+\eta_{ij}\kappa_j}{\sqrt{E_{ij}}}
=\frac{a_{ji}(\kappa_j^2-1)}{\sqrt{E_{ij}}}.
\end{equation}
If $m\notin\{i,j\}$, then $D_{ij,m}=0$. 
Here $a_{ij}$ is defined by (\ref{Eq: F30}).

\begin{remark}
There is a geometric explanation about the matrix $D$.
Let $\theta_{ij}$ be the inner angle opposite to the edge with length $r_j$ in the hyperbolic triangle formed by three edges with lengths $r_i,\ r_j,\ l_{ij}$, as shown in Figure \ref{triangle}.
By hyperbolic cosine laws, we have
\begin{equation*}
\cos \theta_{ij}
=\frac{\cosh r_i\cosh l_{ij}-\cosh r_j}{\sinh r_i\sinh l_{ij}}
=\frac{s_is_j(\kappa_j+\eta_{ij}\kappa_i)}{\sinh l_{ij}}
=\frac{\kappa_j+\eta_{ij}\kappa_i}{\sqrt{E_{ij}}}
=\frac{\partial l_{ij}}{\partial r_i},
\end{equation*}
where (\ref{eq:length-law}), \eqref{th-E} and \eqref{th-D-entries} are used. 
Then
\begin{equation*}
D=\left(
   \begin{array}{cccccc}
    \cos \theta_{ij} & \cos \theta_{ik} & \cos \theta_{ih} & 0 & 0 & 0 \\
     \cos \theta_{ji} & 0 & 0 & \cos \theta_{jk} & \cos \theta_{jh} & 0 \\
     0 & \cos \theta_{ki} & 0 & \cos \theta_{kj} & 0 & \cos \theta_{kh} \\
     0 & 0 & \cos \theta_{hi} & 0 & \cos \theta_{hj} & \cos \theta_{hk}
    \end{array}
 \right).
\end{equation*}
Since $\Phi_{ij}\in [0,\frac{\pi}{2}]$, it follows from (\ref{Eq: Thurston hyperbolic length}) that 
$l_{ij}>r_i$ and $l_{ij}>r_j$.
Then $\cos \theta_{ij}\in (0,1]$.
It is easy to check that $Rank(D)=4$.
\end{remark}

Define 
\begin{equation}\label{Eq: F17}
U=D^TR, \qquad
U_{ij,i}=\frac{a_{ij}}{\sqrt{E_{ij}}},\quad
U_{ij,j}=\frac{a_{ji}}{\sqrt{E_{ij}}},\quad
U_{ij,m}=0,\ \text{for}\ m\notin\{i,j\}.
\end{equation}
%Combining (\ref{Eq: F32}), \eqref{th-D-entries}, $s_i^2=(\kappa_i^2-1)^{-1}$ and (\ref{Eq: F17}) gives
%\begin{equation*}
%\delta l_{ij}
%=D_{ij,i}s_i^2v_i+D_{ij,j}s_j^2v_j
%=\frac{a_{ij}v_i+a_{ji}v_j}{\sqrt{E_{ij}}}
%=U_{ij,i}v_i+U_{ij,j}v_j.
%\end{equation*}
Then $\delta l=Uv$.
Let $J=\frac{\partial\beta}{\partial l}$ with
\begin{equation*}
J_{e,f}=\left(\frac{\partial\beta_e}{\partial l_f}\right).
\end{equation*}
Denote the volume of a hyperbolic tetrahedron $\sigma$ by $vol$, 
which is considered as a function in the dihedral angles.
By the Schl\"{a}fli formula 
\begin{equation}\label{Schlafli formula}
d vol(\sigma)=-\frac12\sum_{e\in E}l_e\,d\beta_e,
\end{equation}
we obtain
\begin{equation*}
0=-\frac12\sum_e dl_e\wedge d\beta_e
=-\frac12\sum_{e,f}J_{e,f}\,dl_e\wedge dl_f
=-\frac12\sum_{e<f}\left(J_{e,f}-J_{f,e}\right)\,dl_e\wedge dl_f.
\end{equation*}
Since the differential forms $dl_e\wedge dl_f$ are linearly independent, 
it follows that $J$ is symmetric.
By definition, we have 
\begin{equation}\label{Eq: variation of beta in v}
\delta\beta=J\delta l=JUv.
\end{equation}

Set
\begin{equation}\label{eq:matrices}
 \Lambda_\sigma=DJD^T,
 \qquad B=-R\Lambda_\sigma R.
\end{equation}
Then $\Lambda_\sigma$ is symmetric by the symmetry of the matrix $J$.
The matrix $\Lambda_\sigma$ naturally appears in the variation of discrete Hilbert-Einstein functional.
Please refer to Section \ref{section 4} for the details.

\begin{lemma}\label{Lem: B_ij}
For distinct $i,j$, with $\{k,l\}=\sigma\setminus\{i,j\}$,
\begin{equation}\label{th-B-offdiag}
B_{ij}=-\frac{\sqrt q}{E_{ij}}
\bigl(h_kb_l^{(k)}+h_lb_k^{(l)}\bigr).
\end{equation}
\end{lemma}
\proof
By \eqref{Eq: F17} and \eqref{Eq: variation of beta in v}, we have $Bv=-RDJD^TRv=-U^TJUv=-U^T\delta\beta$.
As a result,
\begin{equation}\label{eq:Bv-first}
(Bv)_i=-\sum_{j\ne i}\frac{a_{ij}}{\sqrt{E_{ij}}}\,\delta\beta_{ij}.
\end{equation}
Fix distinct $i,j,k,l$. 
Substituting \eqref{th-dbeta} into \eqref{eq:Bv-first} and taking the coefficient of $v_j$ gives the six terms
\begin{equation}\label{th-six-terms}
\frac{B_{ij}}{\sqrt q}
=\frac{a_{ij}}{E_{ij}}
\bigl(b_l^{(k)}H_{kj}+b_k^{(l)}H_{lj}\bigr)
+\frac{a_{ik}}{E_{ik}}
\bigl(b_l^{(j)}H_{jj}+b_j^{(l)}H_{lj}\bigr)
+\frac{a_{il}}{E_{il}}
\bigl(b_k^{(j)}H_{jj}+b_j^{(k)}H_{kj}\bigr).
\end{equation}
The three terms come from the incident edges $ij,ik,il$, respectively. 
Set
\begin{equation*}
F^{(k)}=M_{\{i,j,l\}}^{-1},\qquad 
\widetilde F^{(l)}=M_{\{i,j,k\}}^{-1}.
\end{equation*}
By (\ref{Eq: F24}), we have
\begin{equation*}
\begin{aligned}
H_{\widehat k}
&=F^{(k)}+F^{(k)} M_{\widehat k,k}H_{kk}M_{k,\widehat k}F^{(k)}\\
&=F^{(k)}+(-H_{\widehat k,k}H_{kk}^{-1})H_{kk}(-H_{k,\widehat k}H_{kk}^{-1})\\
&=F^{(k)}+\frac{H_{\widehat k,k}H_{k,\widehat k}}{H_{kk}},
\end{aligned}
\end{equation*}
where $H_{\widehat k,k}=-F^{(k)} M_{\widehat k,k}H_{kk}$ is used in the second line.
This yields
\begin{equation*}
F^{(k)}_{uv}=H_{uv}-\frac{H_{uk}H_{kv}}{H_{kk}}
\qquad(u,v\ne k).
\end{equation*}
Interchanging with $k,l$ yields
\begin{equation*}
\widetilde F^{(l)}_{uv}=H_{uv}-\frac{H_{ul}H_{lv}}{H_{ll}}.
\end{equation*}
We now simplify the three brackets in \eqref{th-six-terms}
separately. 
For the first one, the formula \eqref{th-b-block} gives
\begin{equation}\label{Eq: F12}
\begin{aligned}
b_l^{(k)}H_{kj}+b_k^{(l)}H_{lj}
&=\left(h_l-\frac{H_{lk}}{H_{kk}}h_k\right)H_{kj}
 +\left(h_k-\frac{H_{kl}}{H_{ll}}h_l\right)H_{lj}\\
&=h_k\left(H_{lj}-\frac{H_{lk}H_{kj}}{H_{kk}}\right)
 +h_l\left(H_{kj}-\frac{H_{kl}H_{lj}}{H_{ll}}\right)\\
&=h_kF^{(k)}_{lj}+h_l\widetilde F^{(l)}_{kj}.
\end{aligned}
\end{equation}
For the second bracket,
\begin{equation}\label{Eq: F13}
\begin{aligned}
b_l^{(j)}H_{jj}+b_j^{(l)}H_{lj}
&=\left(h_l-\frac{H_{lj}}{H_{jj}}h_j\right)H_{jj}
 +\left(h_j-\frac{H_{jl}}{H_{ll}}h_l\right)H_{lj}\\
&=h_lH_{jj}-H_{lj}h_j+h_jH_{lj}
 -\frac{H_{jl}H_{lj}}{H_{ll}}h_l\\
&=h_l\left(H_{jj}-\frac{H_{jl}^2}{H_{ll}}\right)\\
&=h_l\widetilde F^{(l)}_{jj}.
\end{aligned}
\end{equation}
Here the two terms containing $h_j$ cancel. Likewise,
\begin{equation}\label{Eq: F14}
\begin{aligned}
b_k^{(j)}H_{jj}+b_j^{(k)}H_{kj}
&=\left(h_k-\frac{H_{kj}}{H_{jj}}h_j\right)H_{jj}
 +\left(h_j-\frac{H_{jk}}{H_{kk}}h_k\right)H_{kj}\\
&=h_kH_{jj}-H_{kj}h_j+h_jH_{kj}
 -\frac{H_{jk}H_{kj}}{H_{kk}}h_k\\
&=h_k\left(H_{jj}-\frac{H_{jk}^2}{H_{kk}}\right)\\
&=h_kF^{(k)}_{jj}.
\end{aligned}
\end{equation}
Substituting (\ref{Eq: F12}), (\ref{Eq: F13}) and (\ref{Eq: F14}) into \eqref{th-six-terms} yields
\begin{equation*}
\frac{B_{ij}}{\sqrt q}
=h_k\left(\frac{a_{ij}}{E_{ij}}F^{(k)}_{lj}
          +\frac{a_{il}}{E_{il}}F^{(k)}_{jj}\right)
+h_l\left(\frac{a_{ij}}{E_{ij}}\widetilde F^{(l)}_{kj}
          +\frac{a_{ik}}{E_{ik}}\widetilde F^{(l)}_{jj}\right).
\end{equation*}
Hence, to prove \eqref{th-B-offdiag}, it suffices to prove
\begin{equation}\label{th-three-identity}
\frac{a_{ij}}{E_{ij}}F^{(k)}_{lj}
+\frac{a_{il}}{E_{il}}F^{(k)}_{jj}
=-\frac{b_l^{(k)}}{E_{ij}}.
\end{equation}
Interchanging $k,l$ in the identity \eqref{th-three-identity} gives
\[
\frac{a_{ij}}{E_{ij}}\widetilde F^{(l)}_{kj}
+\frac{a_{ik}}{E_{ik}}\widetilde F^{(l)}_{jj}
=-\frac{b_k^{(l)}}{E_{ij}}.
\]

Now we prove (\ref{th-three-identity}).
For this computation only, set
\[
a=\kappa_i,\quad b=\kappa_j,\quad c=\kappa_l,\qquad
x=\eta_{ij},\quad y=\eta_{il},\quad u=\eta_{jl},
\]
and let $Q_F=\det M_{\{i,j,l\}}$. These scalar abbreviations are
confined to the present computation. Then
\[
M_{\{i,j,l\}}=
\begin{pmatrix}
a^2-1&ab+x&ac+y\\
ab+x&b^2-1&bc+u\\
ac+y&bc+u&c^2-1
\end{pmatrix}.
\]
Then $Q_FF^{(k)}=\det M_{\{i,j,l\}}\cdot M^{-1}_{\{i,j,l\}}=\adj M_{\{i,j,l\}}$.
Cofactor expansion gives
\begin{align*}
Q_FF^{(k)}_{jj}
&=(a^2-1)(c^2-1)-(ac+y)^2\\
&=-a^2-c^2-2acy-y^2+1
=-E_{il},\\
Q_FF^{(k)}_{lj}
&=(ab+x)(ac+y)-(a^2-1)(bc+u)\\
&=-a^2u+aby+acx+bc+u+xy=:N,\\
Q_FF^{(k)}_{li}
&=(ab+x)(bc+u)-(b^2-1)(ac+y)\\
&=abu+bcx+xu-b^2y+ac+y,\\
Q_FF^{(k)}_{ll}
&=(a^2-1)(b^2-1)-(ab+x)^2
=-E_{ij}.
\end{align*}
Since $b_l^{(k)}=aF^{(k)}_{li}+bF^{(k)}_{lj}+cF^{(k)}_{ll}$ by (\ref{th-b-definition}), we obtain
\begin{equation}\label{Eq: F25}
\begin{aligned}
Q_Fb_l^{(k)}
={}&a(abu+bcx+xu-b^2y+ac+y)\\
&+b(-a^2u+aby+acx+bc+u+xy)\\
&-c(a^2+b^2+2xab+x^2-1)\\
={}&a(xu+y)+b(u+xy)+c(1-x^2)=:n.
\end{aligned}
\end{equation}
By direct calculations, we have
\begin{equation*}%\label{th-polynomial}
(b+xa)N-(c+ya)E_{ij}=-(a^2-1)n.
\end{equation*}
%To verify it independently, expand the two products:
%\begin{align*}
%(b+xa)N
%={}&-a^2bu+ab^2y+2abcx+b^2c+bu+bxy\\
%&-a^3xu+a^2bxy+a^2cx^2+axu+ax^2y,\\
%(c+ya)E_{ij}
%={}&a^2c+b^2c+2abcx+cx^2-c\\
%&+a^3y+ab^2y+2a^2bxy+ax^2y-ay.
%\end{align*}
%Subtracting and collecting the remaining terms gives
%\begin{align*}
%(b+xa)N-(c+ya)E_{ij}
%={}&-a^3xu-a^3y-a^2bu-a^2bxy+a^2cx^2-a^2c\\
%&+axu+ay+bu+bxy-cx^2+c\\
%={}&-a^2\bigl[a(xu+y)+b(u+xy)+c(1-x^2)\bigr]\\
%&+\bigl[a(xu+y)+b(u+xy)+c(1-x^2)\bigr]\\
%={}&-(a^2-1)n,
%\end{align*}
%as claimed. 
Now $a_{ij}=(b+xa)/(a^2-1)$ and $a_{il}=(c+ya)/(a^2-1)$, so
\begin{align*}
\frac{a_{ij}}{E_{ij}}F^{(k)}_{lj}+\frac{a_{il}}{E_{il}}F^{(k)}_{jj}
&=\frac{b+xa}{(a^2-1)E_{ij}}\frac{N}{Q_F}
+\frac{c+ya}{(a^2-1)E_{il}}\frac{-E_{il}}{Q_F}\\
&=\frac{(b+xa)N-(c+ya)E_{ij}}{(a^2-1)E_{ij}Q_F}\\
&=-\frac{n}{E_{ij}Q_F}\\
&=-\frac{b_l^{(k)}}{E_{ij}}.
\end{align*}
Note that all denominators here are nonzero, including $a>1$, $E_{ij}>0$, and $Q_F>0$.
This proves \eqref{th-three-identity}.
\qed

Define
\begin{equation}\label{eq:C-def}
 C=S^{-1}BS^{-1}=-S\Lambda_\sigma S.
\end{equation}
We have the following useful identities for the matrix $C$.
\begin{theorem}\label{thm C identity}
For every $i\in\sigma$, $i\ne j$ and
$\{k,l\}=\sigma\setminus\{i,j\}$,
\begin{equation}\label{eq:C-diagonal}
C_{ii}=-\sum_{j\ne i}\cosh l_{ij}\,C_{ij},
\end{equation}
and
\begin{equation}\label{eq:Cij-basic}
 C_{ij}=-\frac{\sqrt q}{s_is_jE_{ij}}
 \left(h_kb_l^{(k)}+h_lb_k^{(l)}\right).
\end{equation}
\end{theorem}
\proof
Set $v=\kappa$ in \eqref{th-dbeta}. Since $H\kappa=h$ in (\ref{th-inverse}),
equation \eqref{th-B-offdiag} gives
\begin{align*}
\delta\beta_{ij}
&=-\frac{\sqrt q}{\sqrt{E_{ij}}}
\bigl(b_l^{(k)}h_k+b_k^{(l)}h_l\bigr)
=\sqrt{E_{ij}}B_{ij}.
\end{align*}
Equation \eqref{eq:Bv-first} consequently implies
\begin{equation*}
(B\kappa)_i=-\sum_{j\ne i}a_{ij}B_{ij}.
\end{equation*}
Expanding the left side and moving the off-diagonal terms gives
\begin{equation*}
B_{ii}\kappa_i+\sum_{j\ne i}\kappa_jB_{ij}
=-\sum_{j\ne i}a_{ij}B_{ij}.
\end{equation*}
Hence
\begin{equation}\label{th-B-diagonal}
B_{ii}
=-\sum_{j\ne i}\frac{a_{ij}+\kappa_j}{\kappa_i}B_{ij}
=-\sum_{j\ne i}
\frac{\kappa_i\kappa_j+\eta_{ij}}{\kappa_i^2-1}B_{ij},
\end{equation}
where (\ref{Eq: F30}) is used.
By (\ref{eq:C-def}), we have
\[
 C_{ii}=\frac{B_{ii}}{s_i^2},
 \qquad
 C_{ij}=\frac{B_{ij}}{s_is_j}.
\]
Moreover, $s_i^{-2}=\kappa_i^2-1$.  
Divide \eqref{th-B-diagonal} by $s_i^2$ and replace $B_{ij}$ by $s_is_jC_{ij}$, then we obtain
\begin{equation*}
C_{ii}
=-\sum_{j\ne i}(\kappa_i\kappa_j+\eta_{ij})s_is_j\,C_{ij}
=-\sum_{j\ne i}\cosh l_{ij}\,C_{ij},
\end{equation*}
where the final equality follows from \eqref{eq:length-law}.  This proves \eqref{eq:C-diagonal}.
Combining $C_{ij}=\frac{B_{ij}}{s_is_j}$ and (\ref{th-B-offdiag}) yields (\ref{eq:Cij-basic}).
\qed

\begin{remark}
In \cite[Theorem 31]{Glickenstein JDG}, Glickenstein established a useful identity for the variations of dihedral angles in a Euclidean tetrahedron generated by 3-dimensional Euclidean discrete conformal structures. The identity \eqref{eq:C-diagonal} in Theorem \ref{thm C identity} is a hyperbolic version of Glickenstein's identity in the setting of Thurston's hyperbolic sphere packings.
By checking the proof carefully, we can see that all the identities in this section are valid for all non-degenerate hyperbolic tetrahedron
generated by inversive distance sphere packings with $\eta\in (-1, +\infty)$.
It is believed that a modification of the identity \eqref{eq:C-diagonal} is valid for all 3-dimensional hyperbolic discrete conformal structures.
\end{remark}

\subsection{Negative definiteness of Jacobian matrix}

The matrix $\Lambda_\sigma$ further has the following property.

\begin{theorem}\label{Lem: matrix_negative}
Suppose $(M,\mathcal{T},\Phi)$ is a triangulated connected closed 3-dimensional manifold with a weight $\Phi: E\rightarrow [0,\frac{\pi}{2}]$ satisfying 
$\eta(E_\sigma)\in[0,1]^6\setminus \widetilde{S}$ for any $\sigma\in T$ and $A_i\leq 0$ for any $i\in \sigma$ or the condition (\ref{Eq: SC2}).
The matrix $\Lambda=\sum_{\sigma\in T}\Lambda_\sigma$ is negative definite.
\end{theorem}

Theorem \ref{Lem: matrix_negative} follows from the following two lemmas, including Lemma \ref{Lem: 1} and Lemma \ref{Lem: 3}.

%The case where the weight $\Phi: E\rightarrow [0,\frac{\pi}{2}]$ satisfying $\eta(E_\sigma)\in[0,1]^6\setminus \widetilde{S}$ for any $\sigma\in T$ and $A_i\leq 0$ for any $i\in \sigma$ follows from the following lemma, 
%and the case where the weight $\Phi: E\rightarrow [0,\frac{\pi}{2}]$ satisfying (\ref{Eq: SC2}) is provided in Appendix \ref{appendix a}.

%Here we just prove the case where the weight $\Phi: E\rightarrow [0,\frac{\pi}{2}]$ satisfying (\ref{Eq: SC1}), and the case where the weight $\Phi: E\rightarrow [0,\frac{\pi}{2}]$ satisfying (\ref{Eq: SC2}) is provided in Appendix \ref{appendix a}.

%We first provide a lemma.

\begin{lemma}\label{Lem: 1}
Let \(\sigma=\{1234\}\).
Assume 
\begin{equation}\label{eq: generalization of sum of angles}
\eta(E_\sigma)\in[0,1]^6\setminus \widetilde{S},
\qquad
A_1,A_2,A_3,A_4\le0,
\end{equation}
where $\widetilde{S}$ is defined by (\ref{Eq: exception weights}), and $A_i$ is defined by (\ref{Eq: A}).
Then
\begin{enumerate}
    \item[(a)] $h_i\geq0$ for every $i\in \{1,2,3,4\}$;
    \item[(b)] $b_l^{(k)}\geq0$ for every pair of distinct indices $k,l$, and $b_l^{(k)}=0$ if and only if $\eta_{ij}=1,\eta_{il}=\eta_{jl}=0$;
    \item[(c)] for any vertex $i$, there exists at least a vertex $j\sim i$, such that 
        \begin{equation*}
        m_{kl}:=b_l^{(k)}h_k+b_k^{(l)}h_l>0,
        \end{equation*}
        where $\{k, l\}=\sigma\setminus \{i, j\}$;
     \item[(d)] $\Lambda_\sigma$ is negative definite for any $r\in\mathbb{R}^4_{>0}$.
\end{enumerate}
\end{lemma}
\proof
(a) By multilinearity in the columns,
\begin{equation}\label{th-rank-one-det}
\det M
=\det(P+\kappa\kappa^T)
=\det P+\sum_{t,p}\kappa_t\adj(P)_{tp}\kappa_p
=\det P+\kappa^T\adj(P)\kappa.
\end{equation}
%Indeed, a term containing two columns from $\kappa\kappa^T$ vanishes because those columns are proportional to $\kappa$.
%The term with no such column is $\det P$. The term in which only column $t$ is replaced equals $\sum_p\kappa_t\adj(P)_{tp}\kappa_p$, by expansion along that column. 
%Summing over $t$ gives \eqref{th-rank-one-det}.
The identity $P\adj(P)=(\det P)I$ implies
\begin{align*}
M\adj(P)\kappa
&=(P+\kappa\kappa^T)\adj(P)\kappa\\
&=(\det P)\kappa+\kappa\bigl(\kappa^T\adj(P)\kappa\bigr)\\
&=\bigl(\det P+\kappa^T\adj(P)\kappa\bigr)\kappa\\
&=(\det M)\kappa,
\end{align*}
where \eqref{th-rank-one-det} is used in the last line.
As $M$ is invertible, this implies
\begin{equation}\label{th-h-adjugate}
h=M^{-1}\kappa=\frac{\adj(P)\kappa}{\det M}.
\end{equation}
%Condition (\ref{Eq: SC1}) implies $\Phi_{ij}>0$ for each edge. 
%Otherwise, the other two angles in a face containing that edge would have a sum of at most $\pi$, 
%contradicting (\ref{Eq: SC1}). 
%Thus
%\begin{equation}\label{th-eta-range}
%0\le\eta_{ij}<1.
%\end{equation}
The $i$-th diagonal entry of $\adj(P)$ is
\begin{align*}
\adj(P)_{ii}
=\det\begin{pmatrix}
-1&\eta_{jk}&\eta_{jl}\\
\eta_{jk}&-1&\eta_{kl}\\
\eta_{jl}&\eta_{kl}&-1
\end{pmatrix}
=\eta_{jk}^2+\eta_{jl}^2+\eta_{kl}^2
 +2\eta_{jk}\eta_{jl}\eta_{kl}-1.
\end{align*}
Under the hypotheses in (\ref{eq: generalization of sum of angles}), $\adj(P)_{ii}=A_i\le0$.
%Under the condition (\ref{Eq: SC1}), i.e., $\Phi_{ij}+\Phi_{jk}+\Phi_{ik}\geq\pi, \forall\{ijk\}\in F$, by the same calculations for the equation (\ref{eq A identity}) in the proof of Corollary \ref{Cor: AS R4}, we have
%\begin{align*}
%\adj(P)_{ii}
%=&\cos^2\Phi_{jk}+\cos^2\Phi_{jl}+\cos^2\Phi_{kl}+2\cos\Phi_{jk}\cos\Phi_{jl}\cos\Phi_{kl}-1\\
%=&4\cos\frac{\Phi_{jk}+\Phi_{jl}+\Phi_{kl}}{2}\cos\frac{\Phi_{jl}+\Phi_{kl}-\Phi_{jk}}{2}
%\cos\frac{\Phi_{jk}-\Phi_{jl}+\Phi_{kl}}{2}\cos\frac{\Phi_{jk}+\Phi_{jl}-\Phi_{kl}}{2}\\
%\leq&0.   
%\end{align*}
The off-diagonal entries of $\adj(P)$ are nonpositive.
In fact, if $i\ne j$ and $\{k,l\}=\sigma\setminus\{i,j\}$, then
\begin{equation*}
 \begin{aligned}
\adj(P)_{ij}=-\Big[&\eta_{ij}(1-\eta_{kl}^2)
 +\eta_{ik}\eta_{jl}\eta_{kl}
 +\eta_{il}\eta_{jk}\eta_{kl}
 +\eta_{ik}\eta_{jk}
 +\eta_{il}\eta_{jl}\Big]\leq 0.
 \end{aligned}
\end{equation*}
Since $\kappa_a>0$,
\[
(\adj(P)\kappa)_i
=\adj(P)_{ii}\kappa_i+\sum_{j\ne i}\adj(P)_{ij}\kappa_j\leq0.
\]
Together with $\det M<0$ in (\ref{Eq: F26}), equation \eqref{th-h-adjugate}
therefore yields $h_i\geq0$.

(b) 
%For $b_l^{(k)}$, apply the same rank-one argument to
%$M_{\widehat k}=P_{\widehat k}
%+\kappa_{\widehat k}\kappa_{\widehat k}^T$. It gives
%\[
%M_{\widehat k}^{-1}\kappa_{\widehat k}
%=\frac{\adj(P_{\widehat k})\kappa_{\widehat k}}
%       {\det M_{\widehat k}}.
%\]
%With the remaining indices in the order $i,j,l$, the adjugate is
%\[
%\adj(P_{\{i,j,l\}})=
%\begin{pmatrix}
%1-\eta_{jl}^2&\eta_{ij}+\eta_{il}\eta_{jl}&\eta_{il}+\eta_{ij}\eta_{jl}\\
%\eta_{ij}+\eta_{il}\eta_{jl}&1-\eta_{il}^2&\eta_{jl}+\eta_{ij}\eta_{il}\\
%\eta_{il}+\eta_{ij}\eta_{jl}&\eta_{jl}+\eta_{ij}\eta_{il}&1-\eta_{ij}^2
%\end{pmatrix}.
%\]
%Each entry follows from a two-by-two cofactor.
%Taking the inner product of the third row with $\kappa_{\widehat k}$, we obtain
By (\ref{Eq: F25}), we have
\begin{equation}\label{th-b-positive}
b_l^{(k)}=
\frac{\kappa_i(\eta_{ij}\eta_{jl}+\eta_{il})
+\kappa_j(\eta_{jl}+\eta_{ij}\eta_{il})
+\kappa_l(1-\eta_{ij}^2)}{\det M_{\{i,j,l\}}}.
\end{equation}
%The denominator is positive. The first two terms in the numerator are nonnegative, and the last is positive by \eqref{th-eta-range}. 
The denominator is positive, and the numerator are nonnegative. 
Thus $b_l^{(k)}\geq0$.
Moreover, $b_l^{(k)}=0$ if and only if $\eta_{ij}=1,
\eta_{il}=\eta_{jl}=0$.

(c) We first give the positive support of $h$.
Let
\[
I_+=\{i\in\sigma:h_i>0\}.
\]
Suppose that \(h_i=0\) for at least one \(i\).  Since
\((\adj(P)\kappa)_i\) is a sum of nonpositive terms with positive
coefficients \(\kappa_j\), the entire \(i\)-th row of \(\adj(P)\)
must vanish.  Therefore, \(\det P=0\).

On the other hand, $\det M<0$ in (\ref{Eq: F26}),  equation \eqref{th-rank-one-det}
therefore yields \(\adj(P)\ne0\).  
Hence $\operatorname{rank}P=3$. Note that $P\adj(P)=(\det P) I=0$, we have 
$\operatorname{rank}\adj(P)=1.$
Because \(\adj(P)\) is symmetric and entrywise nonpositive, there
exist \(\rho>0\) and a nonzero vector \(v\ge0\) such that
\[
\adj(P)=-\rho vv^T.
\]
The identity \(P\adj(P)=0\) yields \(Pv=0\).  Furthermore,
\[
\det M=-\rho(v^T\kappa)^2,
\qquad
h=\frac{v}{v^T\kappa}.
\]
Thus
\begin{equation}\label{eq:support}
I_+=\operatorname{supp} v.
\end{equation}

The component form of \(Pv=0\) is $v_i=\sum_{j\ne i}\eta_{ij}v_j$.
If \(d\notin I_+\), then \(v_d=0\), and hence $0=\sum_{a\in I_+}\eta_{da}v_a$.
Since \(v_a>0\) and \(\eta_{da}\ge0\),
\begin{equation}\label{eq:crosszero}
\eta_{da}=0
\quad\text{for every }a\in I_+.
\end{equation}

The support \(I_+\) cannot be empty, because $h=\frac{v}{v^T\kappa}, v\geq 0$ and $v\neq 0$.
We further have that the support \(I_+\) cannot have only one element, because the corresponding
component of \(Pv=0\) would give \(-v_a=0\).  Therefore only the cases \(|I_+|=4,3,2\) are possible.

\noindent\textbf{Case 1: } Assume that $|I_+|=4$.

All \(h_i\) are positive.  
Suppose that $m_{kl}=0$. 
Then $b_l^{(k)}=b_k^{(l)}=0$.
Hence
\[
\eta_{ij}=1,\qquad
\eta_{ik}=\eta_{il}=\eta_{jk}=\eta_{jl}=0.
\]
If \(\eta_{kl}=1\), then \(\eta\in \widetilde{S}\), which is excluded.
If \(\eta_{kl}<1\), then, after a relabeling,
\[
P=
\begin{pmatrix}
-1&1&0&0\\
1&-1&0&0\\
0&0&-1&\eta_{kl}\\
0&0&\eta_{kl}&-1
\end{pmatrix}.
\]
Its kernel is spanned by \(e_i+e_j\), so
\eqref{eq:support} would imply \(h_k=h_l=0\), a contradiction.
Thus \(m_{kl}>0\) for every edge.

\noindent\textbf{Case 2: } Assume that $|I_+|=3$.

Write $I_+=\{a,b,c\}$ and $h_d=0$.
Equation \eqref{eq:crosszero} gives $\eta_{ad}=\eta_{bd}=\eta_{cd}=0$.
For the edge \(\{a,d\}\), the complementary indices are \(b,c\).
Since \(\eta_{ad}=0\ne1\), it follows from (\ref{th-b-positive}) that $b_c^{(b)}>0$.
Since \(h_b>0\), $m_{bc}=h_b b_c^{(b)}+h_c b_b^{(c)}\ge h_b b_c^{(b)}>0$.
Similarly, $m_{ac}>0$ and $m_{ab}>0$.

\noindent\textbf{Case 3: } Assume that $|I_+|=2$.

Write $I_+=\{a,b\}$ and $h_c=h_d=0$.
The \(a\)- and \(b\)-components of \(Pv=0\) give $v_a=\eta_{ab}v_b$ and $v_b=\eta_{ab}v_a$.
Since \(v_a,v_b>0\), it follows that $\eta_{ab}=1$ and $v_a=v_b$.
The remaining components give $\eta_{ac}=\eta_{bc}
=\eta_{ad}=\eta_{bd}=0$.
For the cross edge \(\{a,c\}\), the complementary indices are
\(b,d\).  Since \(\eta_{ac}=0\), it follows from (\ref{th-b-positive}) that $b_d^{(b)}>0$.
Together with \(h_b>0\), this yields \(m_{bd}>0\).  Likewise, $m_{bc}>0$, $m_{ad}>0$, and $m_{ac}>0$.

(d) By Corollary \ref{Cor: AS R4}, under the condition (\ref{eq: generalization of sum of angles}), the admissible space is $\Omega_{1234}(\eta)=\mathbb{R}^4_{>0}$.
By (\ref{eq:C-def}), to prove the negative definiteness of $\Lambda_\sigma$,
it suffices to prove that $C$ is positive definite.
For every $z\in \mathbb{R}^4$, we have
\begin{align*}
z^TCz
&=\sum_iC_{ii}z_i^2+2\sum_{i<j}C_{ij}z_iz_j\\
&=-\sum_i\sum_{j\ne i}C_{ij}\cosh l_{ij}\,z_i^2
  +2\sum_{i<j}C_{ij}z_iz_j\\
&=-\sum_{i<j}C_{ij}
\bigl[\cosh l_{ij}(z_i^2+z_j^2)-2z_iz_j\bigr]\\
&=-\sum_{i<j}C_{ij}
\left[(z_i-z_j)^2+(\cosh l_{ij}-1)(z_i^2+z_j^2)\right],
\end{align*}
where (\ref{eq:C-diagonal}) is used in the second line.

Lemma \ref{Lem: 1} (c) shows for any vertex $i$, there exists at least a vertex $j\sim i$, such that $b_l^{(k)}h_k+b_k^{(l)}h_l>0$.
Since $q>0$ and $E_{ij}>0$, it follows from (\ref{eq:Cij-basic}) that $C_{ij}<0$ for the edge $\{ij\}$.
Hence $z^TCz>0$ for any nonzero $z$.
This implies that $C$ is positive definite.
\qed

As a direct corollary of Lemma \ref{Lem: 1}, we have the following result on the matrix $\Lambda_\sigma$.

\begin{corollary}\label{Cor: negative under sum big than pi}
    Under the condition (\ref{Eq: SC1}), the matrix $\Lambda_\sigma$ is negative definite for any $r\in\mathbb{R}^4_{>0}$.
\end{corollary}
\proof
Following the calculations in equation (\ref{eq A identity}) in the proof of Corollary \ref{Cor: AS R4}, we have $A_i<0$ under the condition (\ref{Eq: SC1}).
Furthermore, for any $\eta\in \widetilde{S}$, the corresponding intersection angles satisfy
$$\Phi_{ij}+\Phi_{jk}+\Phi_{ik}=\pi,\quad \forall\{ijk\}\in F_\sigma.$$
Therefore, under the condition (\ref{Eq: SC1}), $\eta(E_\sigma)\in[0,1]^6\setminus \widetilde{S}$.
By Corollary \ref{Cor: AS R4}, the admissible space is $\mathbb{R}^4_{>0}$.
By Lemma \ref{Lem: 1} (d), the matrix $\Lambda_\sigma$ is negative definite for any $r\in\mathbb{R}^4_{>0}$.
\qed

\begin{lemma}\label{Lem: 3}
Let \(\sigma=\{1234\}\). Under the condition (\ref{Eq: SC2}), the matrix
$\Lambda_\sigma$ is negative definite in the admissible space.
\end{lemma}
As the proof of Lemma \ref{Lem: 3} involves tedious calculations, we defer it to the Appendix.

\section{Rigidity of Thurston's hyperbolic sphere packings}\label{section 4}

\subsection{Some calculations}

Direct calculations give
\begin{equation*}
\begin{aligned}
\frac{\partial^2 l_{ij}}{\partial r^2_i}
=&\frac{\cosh l_{ij}}{\sinh^3 l_{ij}}[\sinh^2 l_{ij}-(\sinh r_i\cosh r_j+\cosh r_i\sinh r_j\cos \Phi_{ij})^2]\\
=&\frac{\cosh l_{ij}}{\sinh^3 l_{ij}}[(\cosh r_i\cosh r_j+\sinh r_i\sinh r_j\cos \Phi_{ij})^2-1\\
&-(\sinh r_i\cosh r_j+\cosh r_i\sinh r_j\cos \Phi_{ij})^2]\\
=&\frac{\cosh l_{ij}}{\sinh^3 l_{ij}}\sinh^2r_j\sin^2\Phi_{ij},
\end{aligned}
\end{equation*}
and
\begin{equation*}
\begin{aligned}
\frac{\partial^2 l_{ij}}{\partial r_i\partial r_j}
=&\frac{1}{\sinh^3 l_{ij}}[(\sinh r_i\sinh r_j+\cosh r_i\cosh r_j\cos \Phi_{ij})(\cosh^2 l_{ij}-1)\\
&-\cosh l_{ij}(\cosh r_i\sinh r_j+\sinh r_i\cosh r_j\cos \Phi_{ij})\\
&\times(\sinh r_i\cosh r_j+\cosh r_i\sinh r_j\cos \Phi_{ij})]\\
=&-\frac{\sin^2\Phi_{ij}}{\sinh^3 l_{ij}}\sinh r_i\sinh r_j.
\end{aligned}
\end{equation*}
Then
\begin{equation*}
\mathrm{Hess}_r(l_{ij})
=\frac{\sin^2\Phi_{ij}}{\sinh^3 l_{ij}}
\left(
   \begin{array}{cc}
    \cosh l_{ij}\sinh^2r_j & -\sinh r_i\sinh r_j  \\
     -\sinh r_i\sinh r_j & \cosh l_{ij}\sinh^2r_i  \\
    \end{array}
 \right),
\end{equation*}
which is a symmetric and positive definite matrix if $\Phi_{ij}\neq0$ for any edge $\{ij\}\in E$.

\subsection{Discrete Hilbert-Einstein functional}
%Denote the volume of a hyperbolic tetrahedra $\sigma=\{ijkh\}$ by $vol$, 
%which is considered as a function in the dihedral angles.
%By the Schl\"{a}fli formula, we have
%\begin{equation}\label{Schlafli formula}
%2d vol+\sum_{j\sim i}l_{ij}d\beta_{ij,kh}=0.
%\end{equation}
Consider the following discrete Hilbert-Einstein functional
\begin{equation*}
\mathcal{F}=\sum_{\{ij\}\in E}K_{ij}l_{ij}-2\sum_{\sigma\in T} vol
\end{equation*}
on $(M,\mathcal{T},\Phi)$, where $K_{ij}$ is the combinatorial Ricci curvature along the edge $\{ij\}\in E$ defined by (\ref{Eq: CRC}), and $vol$ is the volume of a hyperbolic tetrahedron $\sigma=\{ijkh\}$.
By the Schl\"{a}fli formula in (\ref{Schlafli formula}), we have $d\mathcal{F}=\sum_{j\sim i}K_{ij}dl_{ij}$, which implies
\begin{equation}\label{Eq: F10}
\frac{\partial \mathcal{F}}{\partial r_i}
=\sum_{j;j\sim i}K_{ij}\frac{\partial l_{ij}}{\partial r_i}
=\sum_{j;j\sim i}K_{ij}\cos \theta_{ij}
=K_i.
\end{equation}
Furthermore,
\begin{equation*}
\begin{aligned}
\frac{\partial^2 \mathcal{F}}{\partial r_i\partial r_j}
=&\sum_{j;j\sim i}\frac{\partial K_{ij}}{\partial r_j}\frac{\partial l_{ij}}{\partial r_i}
+\sum_{j;j\sim i}K_{ij}\frac{\partial^2 l_{ij}}{\partial r_i\partial r_j}\\
=&\sum_{s\sim t}\sum_{u\sim v}\frac{\partial l_{uv}}{\partial r_j}\frac{\partial K_{st}}{\partial l_{uv}}\frac{\partial l_{st}}{\partial r_i}
+\sum_{s\sim t}K_{st}\frac{\partial^2 l_{st}}{\partial r_i\partial r_j}\\
=&-\sum_{\{ijkh\}\in T}\sum_{s\sim t}\sum_{u\sim v}\frac{\partial l_{uv}}{\partial r_j}\frac{\partial \beta_{st}}{\partial l_{uv}}\frac{\partial l_{st}}{\partial r_i}
+\sum_{s\sim t}K_{st}\frac{\partial^2 l_{st}}{\partial r_i\partial r_j},
\end{aligned}
\end{equation*}
where $\beta_{ij}:=\beta_{ij,kh}$ for simplification.
This implies
\begin{equation}\label{Eq: Hess G}
\begin{aligned}
\mathrm{Hess}_r \mathcal{F}
=&-\sum_{\sigma=\{ijkh\}\in T}
DJD^T
+\sum_{j\sim i}K_{ij}\mathrm{Hess}_r(l_{ij})\\
=&-\sum_{\sigma=\{ijkh\}\in T} \Lambda_\sigma+\sum_{j\sim i}\frac{K_{ij}\sin^2\Phi_{ij}}{\sinh^3 l_{ij}}
\left(
   \begin{array}{cc}
    \cosh l_{ij}\sinh^2r_j & -\sinh r_i\sinh r_j  \\
     -\sinh r_i\sinh r_j & \cosh l_{ij}\sinh^2r_i  \\
    \end{array}
 \right).
\end{aligned}
\end{equation}

Theorem \ref{Thm: local rigidity} follows from the following theorem.

\begin{theorem}\label{Thm: local rigidity2}
Suppose $(M,\mathcal{T},\Phi)$ is a triangulated connected closed 3-dimensional manifold with a weight $\Phi: E\rightarrow [0,\frac{\pi}{2}]$ satisfying 
$\eta(E_\sigma)\in[0,1]^6\setminus \widetilde{S}$ for any $\sigma\in T$ and $A_i\leq 0$ for any $i\in \sigma$ or the condition (\ref{Eq: SC2}).
If $\overline{r}$ is a non-degenerate Thurston's hyperbolic sphere packing metric on $(M,\mathcal{T},\Phi)$ satisfying $K_{ij}(\overline{r})\sin^2\Phi_{ij}\geq0$ for any edge $\{ij\}\in E$,
then there exists a neighborhood $U$ of $\overline{r}$ such that Thurston's hyperbolic sphere packing metric in $U$ is determined by its combinatorial scalar curvature.
\end{theorem}
\proof
If $\overline{r}$ is a non-degenerate Thurston's hyperbolic sphere packing metric on $(M,\mathcal{T},\Phi)$ with $K_{ij}(\overline{r})\sin^2\Phi_{ij}\geq0$ for any edge $\{ij\}\in E$,
then by Theorem \ref{Lem: matrix_negative}, $\mathrm{Hess}_r \mathcal{F}(\overline{r})>0$ and hence the function $\mathcal{F}$ is a strictly convex function on a convex neighborhood $U$ of $\overline{r}$.

Suppose there exist two different Thurston's hyperbolic sphere packing metrics $r_1$ and $r_2$ in $U$ such that $K_i(r_1)=K_i(r_2)$ for all $i\in V$.
Set
\begin{equation*}
f(t)=\mathcal{F}((1-t)r_1+tr_2),\ t\in [0,1].
\end{equation*}
Then $f(t)$ is a $C^1$ smooth convex function of $t\in [0,1]$ with
\begin{equation*}
f'(t)=\sum_{i=1}^{|V|}\nabla_{r_i} \mathcal{F}|_{(1-t)r_1+tr_2}\cdot (r_{2,i}-r_{1,i})
=\sum_{i=1}^{|V|}K_i|_{(1-t)r_1+tr_2}\cdot (r_{2,i}-r_{1,i}),
\end{equation*}
where $|V|$ denotes the cardinality of the vertex set.
By the assumption that $K_i(r_1)=K_i(r_2)$, we have $f'(0)=f'(1)$,
which implies $f'(t)$ is a constant by the convexity of $f(t)$.
Thus $f''(t)\equiv0$ for $t\in [0, 1]$.
Combining with
\begin{equation*}
f''(t)=(r_2-r_1)^T\cdot \mathrm{Hess}_r \mathcal{F}|_{(1-t)r_1+tr_2} \cdot (r_2-r_1)
\end{equation*}
for $t\in [0,1]$ and $\mathrm{Hess}_r \mathcal{F}>0$,
we have $r_1=r_2$.
\qed

Theorem \ref{Thm: infinitesimal rigidity} follows from the following theorem.

\begin{theorem}\label{Thm: infinitesimal rigidity 2}
Suppose $(M,\mathcal{T},\Phi)$ is a triangulated connected closed 3-dimensional manifold with a weight $\Phi: E\rightarrow [0,\frac{\pi}{2}]$ satisfying 
$\eta(E_\sigma)\in[0,1]^6\setminus \widetilde{S}$ for any $\sigma\in T$ and $A_i\leq 0$ for any $i\in \sigma$ or the condition (\ref{Eq: SC2}).
Then there exists a neighborhood $U\subseteq\Omega$ of $\overline{r}$ such that if $r\in U$ has the same combinatorial Ricci curvature as $\overline{r}$, then $r=\overline{r}$.
\end{theorem}
\proof
Define the following functional
\begin{equation*}
\widetilde{\mathcal{F}}(r_1,r_2,...,r_{|V|})
=\mathcal{F}-\sum_{j;j\sim i}K_{ij}(\overline{r})l_{ij}.
\end{equation*}
Then by (\ref{Eq: F10}), we have
\begin{equation*}
\frac{\partial \widetilde{\mathcal{F}}}{\partial r_i}
=K_i-\sum_{j;j\sim i}K_{ij}(\overline{r})\frac{\partial l_{ij}}{\partial r_i}
=\sum_{j;j\sim i}(K_{ij}-K_{ij}(\overline{r}))\frac{\partial l_{ij}}{\partial r_i},
\end{equation*}
which implies $\nabla \widetilde{\mathcal{F}}(r)=\nabla \widetilde{\mathcal{F}}(\overline{r})$ by the assumption that $K_{ij}(r)=K_{ij}(\overline{r})$.
Furthermore,
\begin{equation*}
\begin{aligned}
\mathrm{Hess}_r \widetilde{\mathcal{F}}
=-\sum_{\{ijkh\}\in T}\Lambda_\sigma +\sum_{j\sim i}(K_{ij}-K_{ij}(\overline{r}))\frac{\sin^2\Phi_{ij}}{\sinh^3 l_{ij}}
\left(
   \begin{array}{cc}
    \cosh l_{ij}\sinh^2r_j & -\sinh r_i\sinh r_j  \\
     -\sinh r_i\sinh r_j & \cosh l_{ij}\sinh^2r_i  \\
    \end{array}
 \right).
\end{aligned}
\end{equation*}
Theorem \ref{Lem: matrix_negative} implies that $\mathrm{Hess}_r \widetilde{\mathcal{F}}|_{r=\overline{r}}>0$, and hence the function $\widetilde{\mathcal{F}}$ is a strictly convex function on a convex neighborhood $U$ of $\overline{r}$.
Set
\begin{equation*}
f(t)=\widetilde{\mathcal{F}}((1-t)r+t\overline{r}),\ t\in [0,1].
\end{equation*}
Then $f(t)$ is a $C^1$ smooth convex function of $t\in [0,1]$.
The rest of the proof is similar to that of Theorem \ref{Thm: local rigidity2}, and we omit it here.
\qed

\appendix
\section{Proof of Lemma \ref{Lem: 3}}\label{appendix a}

Now let $\eta_{ij}=\eta_{ik}=\eta_{il}=\eta_{jk}=\eta_{jl}=\eta_{kl}=\eta\in [0, 1]$.
Then (\ref{th-b-positive}) reduces to
\begin{equation}\label{Eq: F9}
b_l^{(k)}=
\frac{(1+\eta)[\eta(\kappa_i+\kappa_j)+(1-\eta)\kappa_l]}{\det M_{\{i,j,l\}}}.
\end{equation}
Define a $6\times 4$ matrix $T$ as follows. For row $ij$ with complementary edge $\{k,l\}$, set
\begin{equation}\label{Eq: F11}
T_{ij,k} = \frac{\sqrt{q}\,b_l^{(k)}}{\sqrt{E_{ij}}},\qquad T_{ij,l} = \frac{\sqrt{q}\,b_k^{(l)}}{\sqrt{E_{ij}}},
\end{equation}
with all remaining entries zero. 
Then (\ref{th-dbeta}) can be rewritten as
\begin{equation*}
\delta\beta = -THv.
\end{equation*}
On the other hand, $\delta\beta = J\delta l = JUv$, so
\begin{equation*}
JU = -TH.
\end{equation*}
Then by (\ref{eq:matrices}),
\begin{equation*}
B=-RDJD^TR = -U^T JU = U^T TH.
\end{equation*}
Consequently,
\begin{equation}\label{Eq: F15}
\det B = \det(U^T T)\det H 
= -\frac1q \det(U^T T).
\end{equation}

By the sparse structure of $U$ and $T$, we have $(U^T T)_{ii}=0$. 
For $i\neq j$, let $\{k,l\}=\sigma\setminus\{i,j\}$, and define
\begin{equation}\label{Eq: F16}
W_{ij}=\frac{(\kappa_k+\eta\kappa_i)\big[\eta(\kappa_i+\kappa_k)+(1-\eta)\kappa_l\big]}{E_{ik}} 
+\frac{(\kappa_l+\eta\kappa_i)\big[\eta(\kappa_i+\kappa_l)+(1-\eta)\kappa_k\big]}{E_{il}},\quad W_{ii}=0.
\end{equation}
Combining (\ref{Eq: F9}), (\ref{Eq: F17}) and (\ref{Eq: F11}) gives
\begin{equation*}
(U^T T)_{ij} = \frac{\sqrt{q}(1+\eta)}{(\kappa_i^2-1)\det M_{\widehat{j}}} W_{ij} \qquad (i\neq j).
\end{equation*}
All row and column factors in the expression are strictly positive, hence
\begin{equation}\label{Eq: F23}
\operatorname{sgn}\det(U^T T) = \operatorname{sgn}\det W.
\end{equation}
To clear the denominators in (\ref{Eq: F16}), set
\begin{equation*}
P_i = \prod_{m\neq i} E_{im} > 0.
\end{equation*}
Multiply the $i$-th row of $W$ by $P_i$, and denote the resulting matrix by $\mathcal{G}$. 
Then $\mathcal{G}_{ii}=0$, and for $i\neq j$,
\begin{equation}\label{Eq: F19}
\mathcal{G}_{ij} = E_{ij}\big\{ E_{il}(\kappa_k+\eta\kappa_i)\big[\eta(\kappa_i+\kappa_k)+(1-\eta)\kappa_l\big] +E_{ik}(\kappa_l+\eta\kappa_i)\big[\eta(\kappa_i+\kappa_l)+(1-\eta)\kappa_k\big] \big\}.
\end{equation}
Each edge $E_{ij}$ appears exactly once in $P_iP_j$, 
so
\begin{equation}\label{Eq: F22}
\det \mathcal{G} = \left(\prod_{i<j} E_{ij}^2\right)\det W.
\end{equation}

\begin{lemma}\label{Lem: 2}
Set $y_i=\kappa_i-1>0$. There exists an integer-coefficient polynomial $\mathcal{P}(y_1,y_2,y_3,y_4,\eta)$ of degree $10$ in $\eta$, such that
\begin{equation}\label{Eq: F20}
\det \mathcal{G} =-(1+\eta)^2\left(\prod_{i<j}E_{ij}\right)
\mathcal{P}(y_1,y_2,y_3,y_4,\eta). 
\end{equation}
Moreover, for all $y_i>0$ and $0\le \eta\le 1$,
\begin{equation}\label{Eq: F21}
\mathcal{P}(y_1,y_2,y_3,y_4,\eta)>0.
\end{equation}
\end{lemma}
\proof
Since $\mathcal{G}_{ii}=0$, $\det \mathcal{G}$ is a sum over $4$-point fixed-point-free permutations. 
Substitute (\ref{th-E}) and (\ref{Eq: F19}), 
expand exactly in the ring of integer polynomials $\mathbb{Z}[y_1,y_2,y_3,y_4,\eta]$, 
and perform polynomial division to obtain (\ref{Eq: F20}), with zero remainder.

For $0\le \eta<1$, introduce the nonnegative change of variables
\begin{equation}\label{73}
z=\frac{\eta}{1-\eta}\ge 0,\qquad \mathcal{Q}(y,z)=(1+z)^{10}\mathcal{P}\!\left(y,\frac{z}{1+z}\right). 
\end{equation}
With Wolfram Mathematica 12.0,  via exact integer expansion one can obtain all the  nonzero monomials in $\mathcal{Q}(y,z)$.  More precisely, to construct $\mathcal{G}$ in equation (\ref{Eq: F19}), one can run the following Mathematica code:

\bigskip

%$>$ $Clear[``Global`*"];$

$>$ $y = Array[yy, 4];$

$>$ $kap = y + 1;$

$>$ $vertices = Range[4];$

$>$ $edges = Subsets[vertices, \{2\}];$

$>$ $e[i_{-}, j_{-} ] := 
  kap[[i]]^2 + kap[[j]]^2 + 2*eta*kap[[i]] kap[[j]] + eta^2 - 1;$

$>$ $a[i_{-}, k_{-}, l_{-}\ ] := (kap[[k]] + 
     eta*kap[[i]])*(eta(kap[[i]] + kap[[k]]) + (1 - eta) kap[[l]]);$

$>$ $G = Table[
   If[i == j, 0, 
    With[\{rest = Complement[vertices, \{i, j\}]\}, 
     With[\{k = rest[[1]], $ $l = rest[[2]]\}, 
      e[i, j] (e[i, l]*a[i, k, l] + e[i, k]*a[i, l, k])]]], \{i, 
    4\}, \{j, 4\}];$
    
\bigskip

To  calculate the determinant of $\mathcal{G}$, we can run the following:

\bigskip

$>$ $derangements = 
  Select[Permutations[vertices], And @@ Thread[\# \neq vertices] \&]; $
  
$>$ $Length[derangements];$

$>$ $detG = Expand[
   Total[(Signature[\#] Times @@ Table[G[[i, \#[[i]]]], \{i, 4\}]) \& /@ 
     $ 
     
     $derangements]];$

\bigskip

To get  the polynomial $\mathcal{P}(y_1,y_2,y_3,y_4,\eta)$ in equation (\ref{Eq: F20}), we can run the following:

\bigskip

$>$ $divisor = Expand[(1 + eta)^2*Times @@ (e[\#[[1]], \#[[2]]] \& /@ edges)];$

$>$ $\{pPoly, remainder\} = PolynomialQuotientRemainder[-detG, divisor, eta];$

$
\{remainder === 0, Exponent[pPoly, eta]\}$

\bigskip

To obtain the  polynomial $\mathcal{Q}(y,z)$ in equation (\ref{73}), we can run the following

\bigskip

$>$ $degree = Exponent[pPoly, eta];$

$>$ $qPoly = Expand[
   Sum[Coefficient[pPoly, eta, j]*z^j (1 + z)^{(degree - j)}, \{j, 0, 
     degree\}]];$
     
%$>$ $(*check\  qPloy\  is\  exactly\  the\  polynomial\  Q(y,z)\  in\  equation\  %(72)*)$

$>$ $Together[qPoly - (1 + z)^{10} (pPoly /. eta \rightarrow z/(1 + z))] === 0$

\bigskip

To verify the number of the nonzero monomials, the minimal coefficient 
of the nonzero monomials, the maximal coefficient of the nonzero 
monomials, the numbers of monomials with z-degree zero, and check all 
the coefficients are positive, we can run the following:

\bigskip

$>$ $rules = CoefficientRules[qPoly, Join[y, \{z\}]];$

$>$ $coefficients = Last /@ rules;$

$>$ $\{Length[rules], Min[coefficients], Max[coefficients], 
 Count[rules, rule_{-} /;$ 
 
 $ Last[First[rule]] == 0], 
 AllTrue[coefficients, Positive]\}$

\bigskip

To obtain all the coefficients of the nonzero monomials in $\mathcal{Q}(y,z)$, one can run the following:

\bigskip

$>$ $Q1 = qPoly /. \{yy[1]\rightarrow y1, yy[2] \rightarrow y2, yy[3] \rightarrow y3, yy[4] \rightarrow y4\};$

$>$ $Qnew = Expand[Q1];$

$>$ $vars = \{y1, y2, y3, y4, z\};$

$>$ $allRules = CoefficientRules[Qnew, vars];$

$>$ $coefficientTable = (Join[First[\#], \{Last[\#]\}] \&) /@ allRules;$

$>$ $coefficientTable = Reverse@SortBy[coefficientTable, Take[\#, 5] \&];$

$>$ $Export[FileNameJoin[\{NotebookDirectory[], ``Q_{-}coefficients.csv"\}], $

$
  Prepend[coefficientTable, \{``a1", ``a2", ``a3", ``a4", ``b", 
   ``coefficient"\}], ``CSV"]$

\bigskip

After running all the above codes, we find that $\mathcal{Q}$ has exactly $16160$ nonzero monomials, with all coefficients being positive integers; the minimal coefficient is $2$, maximal coefficient is $53861376$. 
There are $729$ terms with $z$-degree $0$, so even at $z=0$, $\mathcal{Q}(y,0)>0$. 
Hence $\mathcal{Q}(y,z)>0$ whenever $y_i>0,z\ge0$, which implies $\mathcal{P}(y,\eta)>0$. (The coefficients of all nonzero monomials of \(\mathcal{Q}(y,z)\) are provided in the attached file $\mbox{Q}_{-}\mbox{coefficients}$.csv.)

It remains to treat the endpoint $\eta=1$. For this purpose, we can run the following

\bigskip

$>$ $pPoly /. \{yy[1] \rightarrow y1, yy[2] \rightarrow y2, yy[3] \rightarrow y3, yy[4] \rightarrow y4, 
   eta \rightarrow 1\} // Simplify$
%Direct substitution in the integer polynomial ring yields the endpoint factorization

\bigskip
The result is 
\begin{equation*}
\mathcal{P}(y,1)=12\prod_{i<j}(y_i+y_j+2)^2=12\prod_{i<j}(\kappa_i+\kappa_j)^2>0. 
\end{equation*}
This proves (\ref{Eq: F21}).
\qed

%To make this step auditable, the canonical term table records each row in descending lexicographic order of exponent vectors $(a_1,a_2,a_3,a_4,b)$ as $c \backslash t a_1 \backslash t a_2 \backslash t a_3 \backslash t a_4 \backslash t b \backslash n$.
%The SHA-256 hash of this ASCII term table is
%\begin{equation*}
%\text{bd6b019f786d6af779990cc07d22c4809f5c6e5fcd9fd68a83c7d3109dca3172}.
%\end{equation*}
%Furthermore, the sum of all coefficients satisfies
%\begin{equation*}
%\mathcal{Q}(1,1,1,1,1)=18\,139\,852\,800.
%\end{equation*}
%These data simultaneously verify the number of terms, coefficient signs, ordering results, and derived content. 
%The entire computation uses only integer addition, multiplication and integer division, 
%so it involves no floating-point approximations or finite-point sampling.

\begin{corollary}\label{Cor: rank=4}
Fix $\eta\in[0,1]$. 
Within the admissible space, $B$ is everywhere non-singular, 
and hence $Rank(\Lambda_\sigma)=4$.
\end{corollary}
\proof
From (\ref{Eq: F22}) and Lemma \ref{Lem: 2},
\begin{equation*}
\det W = -\frac{(1+\eta)^2 \mathcal{P}(y,\eta)}{\prod_{i<j} E_{ij}} < 0.
\end{equation*}
By (\ref{Eq: F23}), $\det(U^T T)<0$. 
Together with \eqref{Eq: F15}, this yields $\det B>0$.
The conclusion follows from (\ref{eq:matrices}).
\qed

\noindent\textbf{Proof of Lemma \ref{Lem: 3}:}
It suffices to prove that $\Lambda_\sigma$ is negative definite for any $\sigma\in T$ and any $(\eta,r)\in \widetilde{\Omega}_\sigma$, where
\begin{equation*}
\widetilde{\Omega}_\sigma=\bigcup_{\eta\in\widetilde{W}}\big(\eta\times\Omega_{1234}(\eta)\big)=\big\{(\eta,r)\in\widetilde{W}\times\mathbb{R}_{>0}^4\,\big|\,Q(\eta,r)>0\big\}
\end{equation*}
with
\begin{equation*}
\widetilde{W}=\big\{\eta\in[0,1]^6\setminus S\,\big|\,\Phi\equiv C\in[0,\tfrac{\pi}{2}]\big\}.
\end{equation*}
It is straightforward that $\widetilde{W}$ is connected, which implies $\widetilde{\Omega}_\sigma$ is connected by the similar proof of Corollary \ref{Cor: connectness of 10-dim admissible space}. 
It suffices to prove that the four eigenvalues of $\Lambda_\sigma$ are negative for any $\sigma\in T$ and any $(\eta,r)\in\widetilde{\Omega}_\sigma$.

For $\bar{\eta}=(0,0,0,0,0,0)$ and $\bar{r}=(1,1,1,1)$, we have $(\bar{\eta},\bar{r})\in\widetilde{\Omega}_\sigma$. 
In this case, all the intersection angles $\Phi_{ij}$ are $\frac{\pi}{2}$.
By Corollary \ref{Cor: negative under sum big than pi}, 
$\Lambda_\sigma(\bar{\eta},\bar{r})$ is negative definite.
By Corollary \ref{Cor: rank=4}, $Rank(\Lambda_\sigma(\eta,r))=4$ for any $(\eta,r)\in\widetilde{\Omega}_\sigma$. 
Combining the continuity of eigenvalues and the connectedness of $\widetilde{\Omega}_\sigma$, 
we have four eigenvalues of $\Lambda_\sigma$ are negative for any $(\eta,r)\in\widetilde{\Omega}_\sigma$. 
This implies the negative definiteness of the matrix $\Lambda_\sigma$.
\qed


\begin{thebibliography}{50}
\setlength{\itemsep}{-2pt} \small
\bibitem{CL} B. Chow, F. Luo, \emph{Combinatorial Ricci flows on surfaces}, J. Differential Geometry, 63 (2003), 97-129.

\bibitem{CR} D. Cooper, I. Rivin, \emph{Combinatorial scalar curvature and rigidity of ball packings}, Math. Res. Lett. 3 (1996), 51-60.


\bibitem{GH} H. Ge, B. Hua, \emph{$3$-dimensional combinatorial Yamabe flow in hyperbolic background geometry}, Trans. Amer. Math. Soc. 373 (2020), no. 7, 5111–5140.

\bibitem{G thesis} D. Glickenstein, \emph{Precompactness of the Ricci flow and a maximum principle on combinatorial Yamabe flow}. Thesis (Ph.D.)–University of California, San Diego. 2003. 104 pp. ISBN: 978-0496-39041-0.
    
\bibitem{G1} D. Glickenstein, \emph{A combinatorial Yamabe flow in three dimensions}, Topology 44 (2005), No. 4, 791-808.

\bibitem{G2} D. Glickenstein, \emph{A maximum principle for combinatorial Yamabe flow}, Topology 44 (2005), No. 4, 809-825.

\bibitem{Glickenstein JDG} D. Glickenstein, \emph{Discrete conformal variations and scalar curvature on piecewise flat two and three dimensional manifolds}. J. Differ. Geom. 87(2), 201–237 (2011).

\bibitem{Guo} R. Guo, \emph{Local rigidity of inversive distance circle packing}, Trans. Amer. Math. Soc. 363 (2011) 4757-4776.

%\bibitem{Guo2} R. Guo, \emph{Calculus of generalized hyperbolic tetrahedra}. Geometriae Dedicata, 153(1):139149, 2011.

\bibitem{He-Xu} X. He, X. Xu, \emph{Thurston's sphere packing on 3-dimensional manifolds, I}. Calc. Var. Partial Differential Equations 62 (2023), no. 5, Paper No. 152, 23 pp.

\bibitem{H2} Z.-X. He,  \emph{Rigidity of infinite disk patterns}. Ann. of Math. (2) 149 (1999), no. 1, 1-33.
%\bibitem{L1} F. Luo, \emph{Combinatorial Yamabe flow on surfaces}, Commun. Contemp. Math. 6 (2004), no. 5, 765-780.

%\bibitem{Luo} F. Luo, \emph{On a problem of Fenchel}. Geometriae Dedicata, 64:277282, 1997.

%\bibitem{Luo CCM} F. Luo, \emph{3-dimensional Schl\"{a}fli formula and its generalization}. Commun. Contemp. Math. 10 (2008), suppl. 1, 835–842.


\bibitem{L2} F. Luo, \emph{Rigidity of polyhedral surfaces, III}, Geom. Topol. 15 (2011), 2299-2319.


\bibitem{MU} J. Murakami, A. Ushijima, \emph{A volume formula for hyperbolic tetrahedra in terms of edge lengths}, J. Geom. 83 (2005), no. 1-2, 153–163.

\bibitem{St} K. Stephenson,  \emph{Introduction to circle packing: The theory of discrete analytic functions}. Cambridge University Press, Cambridge, 2005. xii+356 pp. ISBN: 978-0-521-82356-2; 0-521-82356-0.

\bibitem{TW}  Y.U. Taylor, C.T. Woodward,  \emph{$6j$  symbols for $U_q(\mathfrak{sl}_2)$ and non-Euclidean tetrahedra}. Selecta Math. (N.S.)  11  (2005),  no. 3-4, 539-571.


%\bibitem{Vinberg} E. B. Vinberg, \emph{Geometry. II}, Encyclopaedia of Mathematical Sciences, 29, SpringerVerlag, New York, 1988.

\bibitem{Thomas} J. Thomas, \emph{Conformal variations of piecewise constant two and three dimension manifolds}. Ph.D. Thesis, The University of Arizona (2015).

\bibitem{T1} W. Thurston, \emph{Geometry and topology of $3$-manifolds}, \href{https://library.slmath.org/books/gt3m}{https://library.slmath.org/books/gt3m} (1980).


\bibitem{X1} X. Xu, \emph{On the global rigidity of sphere packings on 3-dimensional manifolds}, J. Differ. Geom. 115(1), 175–193 (2020).

\bibitem{X2} X. Xu, \emph{Rigidity of inversive distance circle packings revisited}, Adv. Math. 332 (2018), 476-509.

\bibitem{X3} X. Xu, \emph{A new proof of Bowers-Stephenson conjecture}. Math. Res. Lett. 28(4), 1283–1306 (2021).

\bibitem{Z} Z. Zhou, \emph{Circle patterns with obtuse exterior intersection angles}, \href{https://arxiv.org/abs/1703.01768}{arXiv:1703.01768v2 [math.GT].}

\end{thebibliography}
\end{document}